\pdfoutput=1

\IfFileExists{./prepreamble-amspreprint.sty}{\RequirePackage[packages,theorems,changes]{./prepreamble-amspreprint}}{}

\makeatletter
\IfFileExists{./scoop-latex/scoop-packages.sty}{\providecommand*{\input@path}{}\edef\input@path{{./scoop-latex/}\input@path}}{}
\makeatother
\PassOptionsToPackage{style=authoryear-comp,backend=biber}{biblatex}
\PassOptionsToPackage{commentmarkup = footnote}{changes}    

\documentclass[english]{amsart}

\RequirePackage{biblatex}
\RequirePackage{ltxcmds}
\IfFileExists{./preamble-amspreprint.sty}{\RequirePackage[packages,theorems,changes]{./preamble-amspreprint}}{}

\RequirePackage{scoop-local}

\IfFileExists{./postpreamble-amspreprint.sty}{\RequirePackage[packages,theorems,changes]{./postpreamble-amspreprint}}{}

\makeatletter
\@ifpackageloaded{changes}{
\definechangesauthor[name = {Anton Schiela}, color = {TolVibrantRed}]{AS}
\definechangesauthor[name = {Roland Herzog}, color = {TolVibrantOrange}]{RH}
\definechangesauthor[name = {Ronny Bergmann}, color = {TolVibrantMagenta}]{RB}
}{}
\makeatother        

\makeatletter
\@ifpackageloaded{hyperref}{%
	\hypersetup{
		pdftitle = {Nonlinear Fenchel Conjugates},
		pdfauthor = {Anton Schiela, Roland Herzog, Ronny Bergmann},
		pdfkeywords = {Fenchel conjugate, Fenchel duality, non-smooth analysis, convex functions, differentiable manifolds},
		pdfcreator = {Created using the Scoop Template Engine version 1.2.0.post14+git.c1c7e129.dirty.}
	}
}{
	\pdfinfo{
		/Title (Nonlinear Fenchel Conjugates)
		/Author (Anton Schiela, Roland Herzog, Ronny Bergmann)
		/Subject ()
		/Keywords (Fenchel conjugate, Fenchel duality, non-smooth analysis, convex functions, differentiable manifolds)
		/Creator (Created using the Scoop Template Engine version 1.2.0.post14+git.c1c7e129.dirty.)
	}
}
\makeatother

\title[Nonlinear Fenchel Conjugates]{Nonlinear Fenchel Conjugates}

\author[A. Schiela]{Anton Schiela\orcidlink{0000-0002-6959-2951}}
\address[A. Schiela]{Department of Mathematics, University of Bayreuth, 95440 Bayreuth, Germany}
\email{anton.schiela@uni-bayreuth.de}
\urladdr{https://num.math.uni-bayreuth.de/en/team/anton-schiela/}

\author[R. Herzog]{Roland Herzog\orcidlink{0000-0003-2164-6575}}
\address[R. Herzog]{Interdisciplinary Center for Scientific Computing, Heidelberg University, 69120 Heidelberg, Germany}
\address[R. Herzog]{Institute for Mathematics, Heidelberg University, 69120 Heidelberg, Germany}
\email{roland.herzog@iwr.uni-heidelberg.de}
\urladdr{https://scoop.iwr.uni-heidelberg.de}

\author[R. Bergmann]{Ronny Bergmann\orcidlink{0000-0001-8342-7218}}
\address[R. Bergmann]{Norwegian University of Science and Technology, Department of Mathematical Sciences, NO-7491 Trondheim, Norway}
\email{ronny.bergmannn@ntnu.no}
\urladdr{https://www.ntnu.edu/employees/ronny.bergmann}

\date{}

\dedicatory{}

\begin{document}

\begin{abstract}
The classical concept of Fenchel conjugation is tailored to extended real-valued functions defined on linear spaces.
In this paper we generalize this concept to functions defined on arbitrary sets that do not necessarily bear any structure at all.
This generalization is obtained by replacing linear test functions by general nonlinear ones.
Thus, we refer to it as nonlinear Fenchel conjugation.
We investigate elementary properties including the Fenchel-Moreau biconjugation theorem.
Whenever the domain exhibits additional structure, the restriction to a suitable subset of test functions allows further results to be derived.
For example, on smooth manifolds, the restriction to smooth test functions allows us to state the Fenchel-Young theorem for the viscosity Fréchet subdifferential.
On Lie groups, the restriction to real-valued group homomorphisms relates nonlinear Fenchel conjugation to infimal convolution and yields a notion of convexity.

\end{abstract}


\makeatletter
\ltx@ifpackageloaded{hyperref}{%
\subjclass[2010]{\href{https://mathscinet.ams.org/msc/msc2020.html?t=49N15}{49N15}, \href{https://mathscinet.ams.org/msc/msc2020.html?t=90C25}{90C25}, \href{https://mathscinet.ams.org/msc/msc2020.html?t=26B25}{26B25}, \href{https://mathscinet.ams.org/msc/msc2020.html?t=49Q99}{49Q99}}
}{%
}
\makeatother

\maketitle

\section{Introduction}%
\label{section:introduction}

Fenchel duality is a classical concept of (convex) analysis with numerous important applications.
On an arbitrary real linear space~$V$, the Fenchel conjugate~$\conjugate{f}$ of an extended real-valued function~$f \colon V \to \R \cup \set{+\infty, -\infty}$ is defined in purely algebraic terms according to
\begin{equation}
	\label{eq:Fenchel-conjugate:linear-space}
	\algebraicdualSpace{V}
	\ni
	x^*
	\mapsto
	\conjugate{f}(x^*)
	\coloneqq
	\sup_{x \in V} \paren[big]\{\}{x^*(x) - f(x)}
	\in
	\R \cup \set{+\infty, -\infty}
	,
\end{equation}
where $\algebraicdualSpace{V}$ denotes the algebraic dual of~$V$.
In practice, this concept is often particularized to locally convex or normed linear spaces~$V$ and their topological dual~$\dualSpace{V}$.
Results of Fenchel duality comprise, for instance, the Fenchel-Young theorem and the Fenchel-Moreau biconjugation theorem; see, \eg, \cite[Proposition~13.13 and Theorem~13.32]{BauschkeCombettes:2011:1}.

In recent years, Fenchel duality has also become the foundation of many algorithms for convex and non-smooth optimization.
As examples, we mention the family of primal-dual algorithms; see, \eg, \cite{EsserZhangChan:2010:1,Valkonen:2023:1}.
It is an ongoing effort to adapt some of these methods from linear to nonlinear operators as in \cite{Valkonen:2014:1,MomLangerSixou:2022:1}, as well as to nonlinear spaces.
In the latter context, the generalization of the notion of Fenchel conjugate, \eg, to functions defined on smooth manifolds, is of paramount importance.
Several approaches on Riemannian respectively Hadamard manifolds have been developed in \cite{BergmannHerzogSilvaLouzeiroTenbrinckVidalNunez:2021:1,SilvaLouzeiroBergmannHerzog:2022:1,DeCarvalhoBentoNetoMelo:2023:1}.

We also mention a related line of research by \cite{MartinezLegaz:2005:1} to generalizing convexity and conjugation, where two potentially nonlinear sets are set into duality by a coupling function.
We discuss the relation to our work in more detail in \cref{remark:relation-to-work-of-MartinezLegaz}.

\subsection*{Contributions}

In this paper we generalize the classical notion of Fenchel conjugation~\eqref{eq:Fenchel-conjugate:linear-space} to extended real-valued functions~$f$ defined on arbitrary sets~$\cM$ rather than linear spaces~$V$.
A high-level view to~\eqref{eq:Fenchel-conjugate:linear-space} reveals that each element $x^* \in \algebraicdualSpace{V}$ can be viewed as the generator of a probing function $x \mapsto x^*(x) - f(x)$.
The value of $\conjugate{f}(x^*)$ is then understood as the scalar measurement obtained by maximizing that probing function over~$x$.

The simple but far-reaching idea is to replace the linear test function~$x^*$ in the probing function $x^* - f$ by a general nonlinear function~$\varphi \colon \cM \to \R \cup \set{+\infty, -\infty}$.
It turns out that many elementary, mostly algebraic properties of this generalized Fenchel conjugate prevail and do not require any structure on the set~$\cM$ at all.
However, in case $\cM$ does exhibit more structure, refined results can be obtained by restricting the space of test functions to an appropriate space $\cF$ of more regular functions.
For instance, when $\cM$ is a metric space, continuous functions $\cF = C(\cM)$ can be chosen.
When $\cM$ is a smooth manifold, we may consider $\cF = C^k(\cM)$.
And when $\cM$ is a linear space, the classical Fenchel conjugate is recovered when we utilize the algebraic dual space $\cF = \algebraicdualSpace{\cM}$.

\subsection*{Organization}

The remainder of this paper is organized as follows.
Following some preliminary notations in \cref{section:preliminaries}, we introduce a nonlinear Fenchel conjugate in \cref{section:nonlinear-Fenchel-conjugate} and examine its properties when defined on a general set~$\cM$.
\Cref{section:F-regularization-and-biconjugation} generalizes the concept of $\Gamma$-regularization (see, \eg, \cite[Chapter~I.3]{EkelandTemam:1999:1}) by which a function~$f$ is approximated by a pointwise supremum of minorants, as well as the notion of the biconjugate.
\Cref{section:nonlinear-Fenchel-conjugates:manifolds} explores additional properties where $\cM$ is a smooth manifold, enabling a generalization of the well-known Fenchel-Young theorem.
We recover previous definitions of Fenchel conjugates on manifolds as special cases.
In \cref{section:nonlinear-Fenchel-conjugates:groups}, we discuss the case when $\cM$ is a group and obtain a generalized characterization of infimal convolution.
We draw some conclusions in \cref{section:conclusion}.

\section{Preliminaries}%
\label{section:preliminaries}

Let us denote extensions of the real numbers by $\extR \coloneqq \R \cup \set{+\infty}$ and $\ExtR \coloneqq \R \cup \set{+\infty, -\infty}$.
Both sets have a natural total ordering, so $=$, $\ge$, $\le$, $>$, $<$ as well as $\sup$ and $\inf$ over arbitrary subsets (with the usual convention for empty sets) are always well-defined in $\ExtR$ in the straightforward way.
For instance, $-\infty \le a \le +\infty$ holds for all $a \in \ExtR$.
Addition and multiplication in these sets are defined in a natural way.
For instance, $+\infty + a = +\infty$ for $a \in \extR$, $\pm \infty \cdot a = \pm \infty$ for $a > 0$, $\pm \infty \cdot a = \mp \infty$ for $a < 0$.
Subtraction is defined via $a - b \coloneqq a + (-1) \cdot b$.
The expressions $\pm \infty + (\mp \infty)$ and $0 \cdot \pm \infty$, however, are not defined.
Therefore, arithmetic computations in $\ExtR$ have to be performed with some care.
In particular, the addition of $c \in \ExtR$ to both sides of an equality or inequality is an equivalence transformation only when $c \in \R$.
For $c = \pm \infty$, the expression $a + c = b + c$ is either undefined or always true, even when $a \neq b$.

Checking all cases, it can be verified that the three inequalities $a \le b + c$, $a - b \le c$, and $a - c \le b$ are equivalent whenever all expressions are defined.
In case one of the sums is undefined, each of the other two inequality either also contains an undefined term, or it is of the form $\pm \infty \le \pm \infty$ and thus true with equality.

Throughout the paper, $\cM$ is a non-empty set and
\begin{align*}
	\ExtF{\cM}
	&
	\coloneqq
	\set{f \colon \cM \to \ExtR}
	,
	\\
	\extF{\cM}
	&
	\coloneqq
	\set{f \colon \cM \to \extR}
	,
	\\
	\realF(\cM)
	&
	\coloneqq
	\set{f \colon \cM \to \R}
\end{align*}
are sets of (extended) real-valued functions.
Arithmetic operations on $\ExtF{\cM}$ are defined pointwise.
Given $f \in \ExtF{\cM}$, $\alpha f$ is defined for any $\alpha \in \R \setminus \set{0}$.
Moreover, $f + g$ is defined for any $f \in \ExtF{\cM}$ and $g \in \realF(\cM)$, as well as for any $f, g \in \extF{\cM}$, but not in general for $f, g \in \ExtF{\cM}$.
However, the pointwise $\max \set{f, g}$ and $\min \set{f, g}$ are always defined.

To take the difficulty of addition into account, we introduce a domain of definition for sums (and similarly differences) of functions
\begin{equation}
	\label{eq:domain}
	\domain{f + g}
	\coloneqq
	\setDef{x \in \cM}{f(x) + g(x) \text{ is defined}}
	.
\end{equation}
In particular, $\domain{f + g} = \cM$ holds whenever $f \in \realF(\cM)$ or $g \in \realF(\cM)$, or when $f, g \in \extF{\cM}$.
We observe that $\domain{f - f} = \setDef{x \in \cM}{f(x) \in \R}$ is the domain of~$f$.
We also require the restricted domain
\begin{equation}
	\label{eq:restricted-domain}
	\rdomain{f + g}
	\coloneqq
	\setDef{x \in \cM}{f(x) \neq -\infty \text{ and } g(x) \neq -\infty}
	\subseteq
	\domain{f + g}
	.
\end{equation}
This definition excludes points where $f(x) + g(x) = -\infty$ from the domain $\domain{f + g}$.
For $f \in \realF(\cM)$, we introduce
\begin{equation}
	\label{eq:infinity-distance}
	\norm{f}_\infty
	\coloneqq
	\sup_{x \in \cM} \abs{f(x)} \in \interval[normal][]{0}{+\infty}
	.
\end{equation}
It is evident that $\extF{\cM}$ is a convex cone and $\realF(\cM)$ is a linear space.
We introduce the natural partial ordering on $\ExtF{\cM}$ by
\begin{align}
	f \le g
	\quad
	&
	\logeq
	\quad
	f(x) \le g(x)
	\text{ for all }
	x \in \cM
	\notag
	\\
	&
	\mrep[r]{{}\Leftrightarrow{}}{{}\logeq{}}
	\quad
	f(x) - g(x) \le 0
	\text{ for all }
	x \in \domain{f - g}
	.
	\label{eq:partial-ordering-of-functions}
\end{align}

Throughout the paper, we reserve the term \emph{function} to denote mappings with values in the (extended) real numbers.
We denote by $\indicator{\cU}$ the indicator function of a set $\cU \subseteq \cM$, which is zero on $\cU$ and is $+\infty$ otherwise.
Moreover, $\epi f \coloneqq \setDef{(x,c) \in \cM \times \R}{c \ge f(x)}$ represents the epigraph of $f \in \ExtF{\cM}$.
The restriction of any map $A \colon \cM \to \cN$ between any two sets $\cM$, $\cN$, to a subset $\cM_0 \subseteq \cM$, is denoted by $\restr{A}{\cM_0}$.

\section{Nonlinear Fenchel Conjugate on General Sets}
\label{section:nonlinear-Fenchel-conjugate}

With the preliminaries from \cref{section:preliminaries} in place, we define a generalized Fenchel conjugate for extended real-valued functions~$f$ defined on an arbitrary, nonempty set~$\cM$ as follows:
\begin{definition}
	\label{definition:general-nonlinear-Fenchel-conjugate}
	The nonlinear Fenchel conjugate of $f \in \ExtF{\cM}$ is defined as
	\begin{align}
		\genconj{f}
		\colon
		\ExtF{\cM}
		&
		\to
		\ExtR
		\notag
		\\
		\varphi
		&
		\mapsto
		\genconj{f}(\varphi)
		\coloneqq
		\sup \setDef{\varphi(x) - f(x)}{x \in \domain{\varphi - f}}
		.
		\label{eq:general-nonlinear-Fenchel-conjugate}
	\end{align}
\end{definition}
As usual, the supremum in \eqref{eq:general-nonlinear-Fenchel-conjugate} is taken to be $-\infty$ should $\domain{\varphi - f}$ happen to be the empty set.

Before we proceed, we provide three alternative formulations for \eqref{eq:general-nonlinear-Fenchel-conjugate}.
The short proof of their equivalence is given in \cref{section:equivalence-of-alternative-formulations-of-Fenchel-conjugate}.
One alternative way to express \eqref{eq:general-nonlinear-Fenchel-conjugate} is as follows:
\begin{equation}
	\label{eq:general-nonlinear-Fenchel-conjugate:alternative:1}
	\begin{aligned}
		\genconj{f}(\varphi)
		&
		=
		\sup \setDef{\varphi(x) - f(x)}{x \in \rdomain{\varphi - f}}
		\\
		&
		=
		\sup \setDef{\varphi(x) - f(x)}{x \in \cM, \; \varphi(x) \neq -\infty \text{ and } f(x) \neq +\infty}
		.
	\end{aligned}
\end{equation}
Second, in more geometric terms, we also have
\begin{equation}
	\label{eq:general-nonlinear-Fenchel-conjugate:alternative:2}
	\genconj{f}(\varphi)
	=
	\inf \setDef{c \in \R}{\varphi(x) \le f(x)+c \text{ for all } x \in \cM}
	.
\end{equation}
Third, in case we wish to dispense with a domain of definition $\domain{\varphi - f}$ depending on $\varphi$ and $f$, we may write equivalently
\begin{equation}
	\label{eq:general-nonlinear-Fenchel-conjugate:alternative:3}
	\genconj{f}(\varphi)
	=
	\sup \setDef{\varphi(x) - c}{(x,c) \in \epi f}
	.
\end{equation}

\begin{example}
	\label{example:indicator-function}
	The conjugate of the indicator function of a set $\cU \subseteq \cM$ is given by
	\begin{equation*}
		\genconj{\indicator{\cU}}(\varphi)
		=
		\sup \setDef{\varphi(x) - \indicator{\cU}(x)}{x \in \domain{\varphi - \indicator{\cU}}}
		=
		\sup \setDef{\varphi(x)}{x \in \cU}
		.
	\end{equation*}
\end{example}

Compared to the classical Fenchel conjugate~$\conjugate{f}$ that acts on \emph{linear} functions $\cM \to \R$, where $\cM$ is a linear space, the domain $\ExtF{\cM}$ of $\genconj{f}$ is much larger, containing \emph{arbitrary} functions $\cM \to \ExtR$.
The evaluation of the nonlinear Fenchel conjugate in closed form is thus, in general, illusory.
On the other hand, our approach allows for the definition of a Fenchel conjugate also for functions defined on sets $\cM$ that are more general than linear spaces.
As we will see below, most of the well-known properties of classical Fenchel conjugates carry over to the extended domain of nonlinear functions and, moreover, additional simple and useful properties, including some symmetry, become visible.

We re-iterate that in case $\cM$ carries additional algebraic or topological structure and thus admits the definition of more regular functions, we may consider the restriction of the domain of $\genconj{f}$ to an appropriate subset.
Important examples include $\cM$ being a topological space, a smooth manifold, a group or Lie group, or a linear space.
Notably, when $\cM$ is a linear space with algebraic dual~$\algebraicdualSpace{\cM}$, then $\conjugate{f} \coloneqq \restr{\genconj{f}}{\algebraicdualSpace{\cM}}$ is the classical Fenchel conjugate.

\subsection{Algebraic Properties}
\label{subsection:algebraic-properties}

In this section we state and prove some elementary properties of the nonlinear Fenchel conjugate.
\begin{proposition}
	\label{proposition:elementary-properties}
	The following statements hold for any $f, \varphi \in \ExtF{\cM}$.
	\begin{enumerate}
		\item \label[statement]{item:elementary-properties:1}
			The following are equivalent:
			\begin{enumerate}
				\item
					$\genconj{f}(\varphi) = -\infty$.

				\item
					$(\varphi - f)(x) = -\infty$ for all $x \in \domain{\varphi - f}$.

				\item
					$\rdomain{\varphi - f} = \emptyset$

				\item
					$\max \set{-\varphi, f} \equiv +\infty$ on~$\domain{\varphi - f}$.

				\item
					$\max \set{-\varphi, f} \equiv +\infty$ on~$\cM$.
			\end{enumerate}

		\item \label[statement]{item:elementary-properties:2}
			The following are equivalent:
			\begin{enumerate}
				\item
					$\genconj{f}(f) = 0$.

				\item
					$\domain{f - f} \neq \emptyset$.

				\item
					$f(x) \in \R$ for some $x \in \cM$.
			\end{enumerate}
			When these fail to hold, then $\genconj{f}(f) = -\infty$.

		\item \label[statement]{item:elementary-properties:3}
			Moreover, the following symmetry and monotonicity relations hold:
			\begin{subequations}
				\label{eq:elementary-properties}
				\begin{align}
					\genconj{(-f)}(\varphi)
					&
					=
					\genconj{(-\varphi)}(f)
					,
					\label{eq:elementary-properties:1}
					\\
					\genconj f(-\varphi)
					&
					=
					\genconj{\varphi}(-f)
					,
					\label{eq:elementary-properties:2}
					\\
					\varphi
					\le
					f
					\quad
					&
					\Leftrightarrow
					\quad
					\genconj{f}(\varphi)
					\le
					\genconj{f}(f)
					\quad
					\Leftrightarrow
					\quad
					\genconj{f}(\varphi)
					\le
					0
					.
					\label{eq:elementary-properties:3}
				\end{align}
			\end{subequations}
	\end{enumerate}
\end{proposition}
\begin{proof}
	All results are direct consequences of the definition.
\end{proof}

The Fenchel-Young inequality $x^*(x) \le f(x) + \conjugate{f}(x^*)$, which is valid for $f \in \ExtF{V}$ on some linear space~$V$ and $x^* \in \algebraicdualSpace{V}$, is at the heart of classical Fenchel duality theory.
It carries over to the nonlinear case, as long as some care is taken that both sides of the inequality are well-defined.
\begin{theorem}[Fenchel-Young inequality]%
	\label{theorem:Fenchel-conjugate:nonlinear-space:Fenchel-Young-inequality}
	Suppose that $f, \varphi \in \ExtF{\cM}$ and $x \in \cM$.
	\begin{enumerate}
		\item
			The Fenchel-Young inequalities
			\begin{subequations}%
				\label{eq:FYI}
				\begin{align}
					\label{eq:FYI:1}
					\genconj{f}(\varphi)
					&
					\ge
					\varphi(x) - f(x)
					\\
					\label{eq:FYI:2}
					f(x)
					&
					\ge
					\varphi(x) - \genconj{f}(\varphi)
					\\
					\label{eq:FYI:3}
					\varphi(x)
					&
					\le
					f(x) + \genconj{f}(\varphi)
				\end{align}
			\end{subequations}
			hold, provided that the respective right-hand side is defined in $\ExtR$.

		\item
			The following two assertions are equivalent:
			\begin{subequations}%
				\label{eq:FYE:1}
				\begin{align}%
					\label{eq:FYE:1:1}
					\genconj{f}(\varphi)
					&
					=
					\varphi(x) - f(x)
					\\
					\label{eq:FYE:1:2:primal-minimization}
					y
					&
					\mapsto
					\varphi(y) - f(y)
					\text{ attains its supremum }
					\genconj{f}(\varphi)
					\text{ over }
					\domain{\varphi - f}
					\text{ at }
					x
					.
				\end{align}
			\end{subequations}

		\item
			The following assertions are equivalent:
			\begin{subequations}%
				\label{eq:FYE:2}
				\begin{align}
					\label{eq:FYE:2:1}
					f(x)
					&
					 =
					\varphi(x)-\genconj{f}(\varphi)
					\\
					\label{eq:FYE:2:2:dual-minimization}
					\psi
					&
					\mapsto \psi(x) - \genconj{f}(\psi)
					\text{ attains its supremum }
					f(x)
					\text{ over }
					\domain{\psi - \genconj{f}}
					\text{ at }
					\varphi
					.
				\end{align}
			\end{subequations}
	\end{enumerate}
\end{theorem}
In~\eqref{eq:FYE:1}, the function $\varphi - f$ is only defined on $\domain{\varphi - f}$.
At points where $\varphi - f$ is undefined, \eqref{eq:FYE:1:1} cannot hold, so these points are excluded from the maximization in~\eqref{eq:FYE:1:2:primal-minimization}.
Similarly, in~\eqref{eq:FYE:2}, $\cdot \, (x) - \genconj{f}$ is only defined on $\domain{\psi - \genconj{f}}$.
\begin{proof}
	The assertions can be verified by \cref{definition:general-nonlinear-Fenchel-conjugate} of the Fenchel conjugate.
	For instance,~\eqref{eq:FYI:1} follows directly from the definition of $\genconj{f}$ as a supremum, with the right-hand side defined for $x \in \domain{\varphi - f}$.
	Then~\eqref{eq:FYI:2} and~\eqref{eq:FYI:3} (where defined) follow from~\eqref{eq:FYI:1} by simple arithmetic manipulations.
	For the remaining two cases, \ie, $\varphi(x) = f(x) = \pm \infty$,~\eqref{eq:FYI:2} and~\eqref{eq:FYI:3} hold as equalities of the form $\pm \infty = \pm \infty$, provided that the corresponding right-hand sides are defined.
\end{proof}

Next we investigate how the nonlinear Fenchel conjugate behaves under transformations:
\begin{proposition}
	Suppose that $f, g \in \ExtF{\cM}$.
	\begin{enumerate}
		\item
			For $\alpha > 0$ and $\beta \in \R$,
			\begin{equation}
				\label{eq:FDconstant}
				\alpha \genconj{f}(\varphi) + \beta
				 =
				\genconj{(\alpha f)}(\alpha \, \varphi + \beta)
				 =
				\genconj{(\alpha f-\beta)}(\alpha \, \varphi)
				.
			\end{equation}

		\item
			If $\domain{f - \psi} = \domain{\varphi + \psi} = \cM$, then
			\begin{equation}
				\label{eq:FDSum1}
				\genconj{(f - \psi)}(\varphi)
				 =
				\genconj{f}(\varphi+\psi)
				.
			\end{equation}

		\item
			If $\domain{f + g} = \domain{\varphi + \psi} = \cM$ and $\genconj{f}(\varphi) + \genconj{g}(\psi)$ is defined, then
			\begin{equation}
				\label{eq:FDSum2}
				\genconj{(f+g)}(\varphi+\psi)
				\le
				\genconj{f}(\varphi) + \genconj{g}(\psi)
				.
			\end{equation}
	\end{enumerate}
\end{proposition}
\begin{proof}
	\Cref{eq:FDconstant} follows from
	\begin{equation*}
		\begin{aligned}
			\alpha \, (\sup [\varphi(x) - f(x)]) + \beta
			&
			=
			\sup [(\alpha \, \varphi(x) + \beta) - \alpha f(x)]
			\\
			&
			=
			\sup [ \alpha \, \varphi(x) - (\alpha f(x) - \beta)]
			,
		\end{aligned}
	\end{equation*}
	where in each case, the $\sup$ extends over $x \in \domain{\varphi - f}$, which is precisely the set where each expression is defined.
	Similarly,~\eqref{eq:FDSum1} follows from
	\begin{equation*}
		\sup [(\varphi + \psi)(x) - f(x)]
		=
		\sup [\varphi(x) - (f - \psi)(x)]
	\end{equation*}
	with the $\sup$ extending over $\rdomain{(\varphi + \psi) - f} = \rdomain{\varphi - (f - \psi)}$; \cf \eqref{eq:restricted-domain}.
	The equality of these restricted domains follows from the assumption $\domain{f - \psi} = \domain{\varphi + \psi} = \cM$.
	Indeed, one verifies that $\varphi(x) = +\infty$ or $f(x) - \psi(x) = -\infty$ holds if and only if $\varphi(x) + \psi(x) = +\infty$ or $f(x) = -\infty$.
	Finally, \eqref{eq:FDSum2} is obtained from
	\begin{equation*}
		\sup [(\varphi + \psi)(x) - (f + g)(x)]
		\le
		(\sup [\varphi(x) - f(x)])
		+
		(\sup [\psi(x) - g(x)])
		,
	\end{equation*}
	with the $\sup$ extending over $\rdomain{(\varphi + \psi) - (f + g)}$ and its supersets $\rdomain{\varphi - f}$ and $\rdomain{\psi - g}$, respectively.
\end{proof}

We conclude this section with results on monotonicity, Lipschitz continuity and convexity of the nonlinear Fenchel conjugate:
\begin{proposition}
	\label{proposition:monotonicity-Lipschitz-continuity-convexity}
	Suppose that $f, g \in \ExtF{\cM}$.
	\begin{enumerate}
		\item \label[statement]{item:monotonicity-Lipschitz-continuity-convexity:1}
			$\varphi \ge \psi$ and $f \le g$ implies $\genconj{f}(\varphi) \ge \genconj{g}(\psi)$.
			In particular, $\genconj{f}$ is monotone increasing.

		\item \label[statement]{item:monotonicity-Lipschitz-continuity-convexity:2}
			$\genconj{f}$ satisfies the following Lipschitz estimate for $\varphi,\psi \in \realF(\cM)$:
			\begin{equation*}
				\abs{\genconj{f}(\varphi) - \genconj{f}(\psi)}
				\le
				\norm{\varphi - \psi}_\infty
				.
			\end{equation*}

		\item \label[statement]{item:monotonicity-Lipschitz-continuity-convexity:3}
			$\genconj{f}$ is convex on $\extF{\cM}$.
	\end{enumerate}
\end{proposition}
\begin{proof}
	To prove \cref{item:monotonicity-Lipschitz-continuity-convexity:1}, we first observe that our assumptions imply $\rdomain{\psi - g} \subseteq \rdomain{\varphi - f}$ and also
	\begin{equation*}
		\psi(x) - g(x)
		\le
		\varphi(x) - f(x)
		\quad
		\text{for all }
		x \in \rdomain{\psi - g}
		.
	\end{equation*}
	Hence we find
	\begin{equation*}
		\genconj g(\psi)
		=
		\sup_{x \in \rdomain{\psi - g}} \psi(x) - g(x)
		\le
		\sup_{x \in \rdomain{\varphi - f}} \varphi(x) - f(x)
		=
		\genconj f(\varphi)
		.
	\end{equation*}
	For \cref{item:monotonicity-Lipschitz-continuity-convexity:2}, our assumptions imply $\domain{f-\varphi} = \domain{f-\psi} = \cM$ and thus
	\begin{equation*}
		\begin{aligned}
			\genconj f(\varphi)
			&
			=
			\sup_{x \in \cM} \varphi(x) - f(x)
			\\
			&
			=
			\sup_{x \in \cM} [\psi(x) - f(x)] + [\varphi(x) - \psi(x)]
			\\
			&
			\le
			\genconj f(\psi) + \norm{\varphi - \psi}_\infty
			.
		\end{aligned}
	\end{equation*}
	In a similar way, we conclude $\genconj f(\psi) \le \genconj f(\varphi) + \norm{\varphi - \psi}_\infty$.

	To prove \cref{item:monotonicity-Lipschitz-continuity-convexity:3}, fix a value $\lambda \in \interval(){0}{1}$.
	We seek to verify the prerequisites for \eqref{eq:FDSum2}.
	We first observe that $\domain{f} = \domain{\lambda f + (1-\lambda) f} = \domain{\varphi + \psi} = \cM$ holds, since $\varphi,\psi \in \extF{\cM}$.
	Moreover, $\genconj f(\varphi) + \genconj f(\psi)$ fails to be defined only if $\genconj f(\varphi) = -\infty$ and $\genconj f(\psi) = +\infty$ or vice versa.
	Considering \wolog the former case and taking into account $\varphi \in \extF{\cM}$, \cref{proposition:elementary-properties}~\ref{item:elementary-properties:1} implies $f \equiv +\infty$.
	Consequently, $\genconj f(\psi)$ is also $-\infty$, and we conclude that $\genconj f(\varphi) + \genconj f(\psi)=-\infty$ is defined.
	Indeed, by the same arguments, $\genconj{(\lambda f)}(\lambda \varphi) + \genconj{((1-\lambda) f)} ((1-\lambda) \psi)$ is defined.
	We now conclude
	\begin{align*}{2}
		\genconj f(\lambda \varphi + (1-\lambda) \psi)
		&
		 =
		\genconj{(\lambda f + (1-\lambda) f)}(\lambda \varphi + (1-\lambda) \psi)
		\\
		&
		\le
		\genconj{(\lambda f)}(\lambda \varphi) + \genconj{((1-\lambda) f)} ((1-\lambda) \psi)
		&
		&
		\quad
		\text{by \eqref{eq:FDSum2}}
		\\
		&
		=
		\lambda \genconj f(\varphi) + (1-\lambda) \genconj f(\psi)
		&
		&
		\quad
		\text{by \eqref{eq:FDconstant}}
		.
		\qedhere
	\end{align*}
\end{proof}

\subsection{Nonlinear Dual Maps}%
\label{subsection:nonlinear-dual-maps}

Rules for the conjugates of functions involving duals of linear operators are a central element in classical Fenchel theory.
As we will see, similar rules hold for nonlinear Fenchel conjugates, provided that the notion of dual map is extended appropriately.
For the purpose of this section, suppose that $\cM$ and $\cN$ are two non-empty sets and $A \colon \cM \to \cN$ is some mapping.

\begin{definition}
	\label{definition:dual-map}
	The map
	\begin{align}
		\genadj{A}
		\colon
		\ExtF{\cN}
		&
		\to
		\ExtF{\cM}
		\notag
		\\
		\psi
		&
		\mapsto
		\genadj{A}(\psi)
		\coloneqq
		\psi \circ A
		\label{eq:dual-mapping}
	\end{align}
	is called the dual or adjoint map of~$A$, or the pullback by~$A$.
\end{definition}

Notice that
\begin{equation*}
	\genadj{A}(\alpha \, \psi_1 + \psi_2)
	=
	(\alpha \, \psi_1 + \psi_2) \circ A
	=
	\alpha \, (\psi_1 \circ A) + \psi_2 \circ A
	=
	\alpha \genadj{A}(\psi_1) + \genadj{A}(\psi_2)
\end{equation*}
holds whenever all terms are defined.
In particular, $\genadj{A} \colon \realF(\cN) \to \realF(\cM)$ is a homomorphism of linear spaces.
When $c \colon \cN \to \R$ is a constant function, then $\genadj{A}(c) = c \colon \cM \to \R$ is the mapping to the same constant, implying the formula $\genadj{A}(\varphi + c) = \genadj{A}(\varphi) + c$.

Quite obviously, $\genadj{A}$ is monotone with respect to the ordering \eqref{eq:partial-ordering-of-functions}, since
\begin{equation*}
	\varphi
	\le
	\psi
	\quad
	\Rightarrow
	\quad
	\genadj{A}(\varphi)
	 =
	\varphi \circ A
	\le
	\psi \circ A
	 =
	\genadj{A}(\psi)
\end{equation*}
holds.
When $A$ is surjective, we have equivalence.
When $\cM$ and $\cN$ are linear spaces and $A \colon \cM \to \cN$ is linear, then $\adj{A} \coloneqq \restr{\genadj{A}}{\algebraicdualSpace{\cN}} \colon \algebraicdualSpace{\cN} \to \algebraicdualSpace{\cM}$ is the classical dual operator.

Next we relate the conjugate of a composition with the composition of conjugates.
\begin{proposition}
	\label{proposition:conjugate-of-composition}
	Suppose that $f \in \ExtF{\cM}$ and $A \colon \cM \to \cN$.
	\begin{enumerate}
		\item \label[statement]{item:conjugate-of-composition:1}
			If $A$ is bijective with inverse $A^{-1}$, then
			\begin{equation*}
				\genconj{(f \circ A^{-1})}
				=
				\genconj{f} \circ \genadj{A}
				.
			\end{equation*}

		\item \label[statement]{item:conjugate-of-composition:2}
			More generally, if we define $(f \bullet A^{-1})(y) \coloneqq \inf_{x \in A^{-1}(y)} f(x)$, we obtain
			\begin{equation*}
				\genconj{(f \bullet A^{-1})}
				=
				\genconj{f} \circ \genadj{A}
				.
			\end{equation*}
	\end{enumerate}
\end{proposition}
\begin{proof}
	\Cref{item:conjugate-of-composition:1} is a special case of \cref{item:conjugate-of-composition:2} so we only address the latter.
	Suppose $\varphi \in \ExtF{\cM}$.
	The result follows from the computation
	\begin{equation*}
		\begin{aligned}
			\MoveEqLeft
			(\genconj{f} \circ \genadj{A})(\varphi)
			=
			\genconj{f}(\varphi \circ A)
			\\
			&
			=
			\sup \setDef[big]{\varphi(A(x)) - f(x)}{x \in \rdomain{\varphi \circ A - f}}
			\quad
			\text{by \eqref{eq:general-nonlinear-Fenchel-conjugate:alternative:1}}
			\\
			&
			=
			\sup \setDef[Big]{\varphi(y) - \inf_{x \in A^{-1}(y)} f(x)}{%
				y \in \range A
				,
				\varphi(y) \neq -\infty
				\text{ and }
				\inf_{x \in A^{-1}(y)} \neq +\infty
			}
			\\
			&
			=
			\sup \setDef[Big]{\varphi(y) - \inf_{x \in A^{-1}(y)} f(x)}{%
				y \in \cN
				,
				\varphi(y) \neq -\infty
				\text{ and }
				\inf_{x \in A^{-1}(y)} \neq +\infty
			}
			\\
			&
			=
			\sup \setDef[big]{\varphi(y) - (f \bullet A^{-1})(y)}{y \in \rdomain{\varphi - f \bullet A^{-1}}}
			\\
			&
			=
			\genconj{(f \bullet A^{-1})}(\varphi)
			.
		\end{aligned}
	\end{equation*}
	The fourth equality holds since $y \not \in \range A$ implies $A^{-1}(y) = \emptyset$ and thus we have $\inf \setDef{f(x)}{x \in A^{-1}(y)} = +\infty$.
	Therefore, we may replace $y \in \range A$ by $y \in \cN$.
\end{proof}

\subsection{Continuity Properties}%
\label{subsection:continuity-properties}

A couple of things can be said about continuity properties of $\genconj{f}$.
In the following we consider the pointwise supremum of functions over a (possibly empty) index set $\Lambda$, $(\sup_{\lambda \in \Lambda} \varphi_\lambda)(x) \coloneqq \sup \setDef{\varphi_\lambda(x)}{\lambda \in \Lambda}$, and analogously pointwise infima, limits of sequences and so on.

\begin{proposition} \hfill
	\label{proposition:continuity-of-generalized-conjugation}
	\begin{enumerate}
		\item \label[statement]{item:continuity-of-generalized-conjugation:1}
			Suppose that $\Lambda$ is an index set and $\setDef{\varphi_\lambda}{\lambda \in \Lambda} \subseteq \ExtF{\cM}$.
			Then
			\begin{align*}
				\varphi
				=
				\sup_{\lambda \in \Lambda} \varphi_\lambda
				\quad
				\Rightarrow
				\quad
				\genconj{f}(\varphi)
				&
				=
				\sup_{\lambda \in \Lambda} \genconj{f}(\varphi_\lambda)
				\\
				\varphi
				=
				\inf_{\lambda \in \Lambda} \varphi_\lambda
				\quad
				\Rightarrow
				\quad
				\genconj{f}(\varphi)
				&
				\le
				\inf_{\lambda \in \Lambda} \genconj{f}(\varphi_\lambda)
				.
			\end{align*}
			Equality holds if, for each $\varepsilon > 0$, there exists $\lambda(\varepsilon)$ such that $\varphi_{\lambda(\varepsilon)} \le \varphi + \varepsilon$.

		\item \label[statement]{item:continuity-of-generalized-conjugation:2}
			For any sequence ${(\varphi_k)}_{k \in \N}$ in $\ExtF{\cM}$, we have
			\begin{equation*}
				\varphi(x)
				=
				\liminf_{k \to \infty} \varphi_k(x)
				\text{ for all }
				x \in \cM
				\quad
				\Rightarrow
				\quad
				\genconj{f}(\varphi)
				\le
				\liminf_{k \to \infty} \genconj{f}(\varphi_k)
				.
			\end{equation*}
			Equality holds if there exists a non-negative real-valued sequence $(\varepsilon_k)_{k \in \N}$ such that $\varphi_k \le \varphi + \varepsilon_k$ with $\liminf_{k \to \infty} \varepsilon_k = 0$.

		\item \label[statement]{item:continuity-of-generalized-conjugation:3}
			In particular, $\genconj{f}$ is lower semi-continuous with respect to pointwise convergence and continuous with respect to pointwise monotone increasing convergence and with respect to uniform convergence.
	\end{enumerate}
\end{proposition}
\begin{proof}
	We obtain the first part of \cref{item:continuity-of-generalized-conjugation:1} using~\eqref{eq:general-nonlinear-Fenchel-conjugate:alternative:3} and an exchange of suprema:
	\begin{equation*}
		\genconj{f}(\varphi)
		 =
		\sup_{(x,c) \in \epi f} (\sup_{\lambda \in \Lambda} \varphi_\lambda(x) - c)
		 =
		\sup_{\lambda \in \Lambda} \sup_{(x,c) \in \epi f}(\varphi_\lambda(x) - c)
		 =
		\sup_{\lambda \in \Lambda} \genconj{f}(\varphi_\lambda)
		.
	\end{equation*}
	To prove the second part, we notice that $\varphi \le \varphi_\lambda$ for all $\lambda \in \Lambda$ implies $\genconj f(\varphi) \le \genconj f(\varphi_\lambda)$ by \cref{proposition:monotonicity-Lipschitz-continuity-convexity}, and thus $\genconj f(\varphi) \le \inf_{\lambda \in \Lambda} \genconj f(\varphi_\lambda)$.
	For the reverse inequality, we fix $\varepsilon > 0$ and let $\lambda(\varepsilon) \in \Lambda$ denote an index such that $\varphi_{\lambda(\varepsilon)} \le \varphi + \varepsilon$.
	Then by \cref{proposition:monotonicity-Lipschitz-continuity-convexity} and \eqref{eq:FDconstant},
	\begin{equation*}
		\inf_{\lambda \in \Lambda} \genconj{f}(\varphi_\lambda)
		\le
		\genconj{f}(\varphi_{\lambda(\varepsilon)})
		\le
		\genconj{f}(\varphi+\varepsilon)
		 =
		\genconj{f}(\varphi)+\varepsilon
		.
	\end{equation*}
	Thus, $\inf_{\lambda \in \Lambda} \genconj{f}(\varphi_\lambda)=\genconj{f}(\varphi)$ holds in this case.

	For \cref{item:continuity-of-generalized-conjugation:2}, we proceed as in the proof of Fatou's Lemma:
	\begin{equation*}
		\begin{aligned}
			\genconj{f}(\liminf_{k \to \infty} \varphi_k)
			&
			=
			\genconj{f}(\adjustlimits \sup_{n \in \N} \inf_{k \ge n} \varphi_k)
			\\
			&
			=
			\sup_{n \in \N} \genconj{f}(\inf_{k \ge n} \varphi_k)
			&
			&
			\text{by \cref{item:continuity-of-generalized-conjugation:1}}
			\\
			&
			\le
			\adjustlimits \sup_{n \in \N} \inf_{k \ge n} \genconj{f}(\varphi_k)
			&
			&
			\text{by \cref{item:continuity-of-generalized-conjugation:1}}
			\\
			&
			=
			\liminf_{k \to \infty} \genconj{f}(\varphi_k)
			.
		\end{aligned}
	\end{equation*}
	For the reverse inequality, suppose that $\varepsilon_k \searrow 0$ is a real-valued sequence satisfying $\varphi_k \le \varphi + \varepsilon_k$.
	We have
	\begin{equation*}
		\begin{aligned}
			\liminf_{k \to \infty} \genconj{f}(\varphi_k)
			&
			\le
			\liminf_{k \to \infty} \genconj{f}(\varphi + \varepsilon_k)
			&
			&
			\text{by \cref{proposition:monotonicity-Lipschitz-continuity-convexity}}
			\\
			&
			=
			\genconj{f}(\varphi) + \liminf_{k \to \infty} \varepsilon_k
			&
			&
			\text{by \eqref{eq:FDconstant}}
			\\
			&
			=
			\genconj{f}(\varphi)
			.
		\end{aligned}
	\end{equation*}
	Thus, in this case, we obtain equality.

	\Cref{item:continuity-of-generalized-conjugation:3} is a direct consequence of \cref{item:continuity-of-generalized-conjugation:2}: replacing the pointwise $\liminf$ of $\varphi_k$ by the pointwise limit, we directly obtain lower semi-continuity of $\genconj{f}$.
	If $\varphi_k$ converges monotonically, we can choose $\varepsilon_k = 0$.
	In case $\varphi_k$ converges uniformly, we have $\varepsilon_k \to 0$ and achieve the proposed continuity statement.
\end{proof}

\section{\texorpdfstring{$\cF$}{𝓕}-Regularization and \texorpdfstring{$\cF$}{𝓕}-Biconjugates}
\label{section:F-regularization-and-biconjugation}

A central topic in classical Fenchel duality is the approximation of a function~$f$ by its maximal convex and lower semi-continuous minorant.
This extremal minorant, here denoted by~$\regularization{f}$, can be constructed as the pointwise supremum of all minorants of~$f$ from the particular class~$\cF$ of continuous affine functions.
This process has been introduced under the name $\Gamma$-regularization in \cite[Chapter~I.3]{EkelandTemam:1999:1} for functions defined on locally convex linear topological spaces.
The biconjugation theorem states that the $\Gamma$-regularization of $f$ coincides with the biconjugate of~$f$.

In this section we generalize the ideas and results of $\Gamma$-regularization and biconjugation from the classical setting of functions on locally convex linear topological spaces to functions on general sets, and from the classical Fenchel to the new nonlinear Fenchel conjugate.
Our key idea is to consider sets~$\cF$ of functions more general than affine, from which the minorants of~$f$ can be taken.
This leads us to the notions of $\cF$-regularization (\cref{subsection:F-regularization}) and $\cF$-biconjugates (\cref{subsection:F-biconjugation}).

\subsection{\texorpdfstring{$\cF$}{𝓕}-Regularization}
\label{subsection:F-regularization}

Suppose that $\cF \subseteq \ExtF{\cM}$ is a non-empty set of functions and denote by
\begin{equation}\label{eq:FTilde}
	\affine{\cF}
	\coloneqq
	\setDef{\varphi + c}{\varphi \in \cF, \; c \in \R}
\end{equation}
the set of all $\varphi$ that result from a shift of elements of $\cF$ by an arbitrary constant.

We define the $\cF$-regularization of $f \in \ExtF{\cM}$ as
\begin{equation}
	\label{eq:defreg}
	\regularization{f}(x)
	\coloneqq
	\sup \setDef[big]{\varphi(x)}{\varphi \in \affine{\cF}, \; \varphi \le f}
	.
\end{equation}
That is, $\regularization{f}$ is the pointwise supremum of all minorants of~$f$ taken from the set~$\cF$ and its constant shifts.
We recall that suprema over sets of functions are understood in the pointwise sense.
We may therefore write \eqref{eq:defreg} more concisely as $\regularization{f} \coloneqq \sup \setDef[big]{\varphi \in \affine{\cF}}{\varphi \le f}$.

Notice that $\regularization{f} \equiv -\infty$ holds for $f \in \ExtF{\cM}$ in case no function $\varphi \le f$ exists in $\affine{\cF}$.
In the important case that $\cF \subseteq \extF{\cM}$ we have either $\regularization{f} \equiv -\infty$ or $\regularization{f} \in \extF{\cM}$.

\begin{proposition}
	\label{proposition:elementary-properties-of-F-regularization}
	Suppose that $f, g \in \ExtF{\cM}$ and that $\cF, \cG \subseteq \ExtF{\cM}$ are any sets of functions.
	Then
	\begin{enumerate}
		\item \label[statement]{item:elementary-properties-of-F-regularization:1}
			$f \le g$ and $\cF \subseteq \cG$ implies $\regularization{f} \le \regularization[\cG]{g}$.

		\item \label[statement]{item:elementary-properties-of-F-regularization:2}
			For $\varphi \in \cF$ and $c \in \R$ we have $\regularization{f+\varphi+c} = \regularization{f} + \varphi+c$.

		\item \label[statement]{item:elementary-properties-of-F-regularization:3}
			$\regularization{f} \le f$, thus $f \le \regularization{f} \Leftrightarrow \regularization{f} = f$

		\item \label[statement]{item:elementary-properties-of-F-regularization:4}
			$f \in \cF \Rightarrow \regularization{f}=f$.

		\item \label[statement]{item:elementary-properties-of-F-regularization:5}
			$\cF \subseteq \cG$ implies $\regularization{\regularization[\cG]{f}} = \regularization{f}$.
	\end{enumerate}
\end{proposition}
\begin{proof}
	\Cref{item:elementary-properties-of-F-regularization:1,item:elementary-properties-of-F-regularization:2,item:elementary-properties-of-F-regularization:3,item:elementary-properties-of-F-regularization:4} follow immediately from the definition.
	We only prove \cref{item:elementary-properties-of-F-regularization:5}.
	Consider an arbitrary $\varphi \in \affine{\cF}$ such that $\varphi \le f$.
	Then by definition of $\regularization{f}$ as a supremum, we have $\varphi \le \regularization{f} \le \regularization[\cG]{f}$.
	Thus the set of $\affine{\cF}$-minorants of~$f$ is contained in the set of $\affine{\cF}$-minorants of $\regularization[\cG]{f}$.
	This implies $\regularization{f} \le \regularization{\regularization[\cG]{f}}$ and thus $\regularization{f} = \regularization{\regularization[\cG]{f}}$ by \cref{item:elementary-properties-of-F-regularization:1,item:elementary-properties-of-F-regularization:3}.
\end{proof}

In case $\cF$ enjoys additional structure, then the $\cF$-regularization also may have additional properties:
\begin{proposition}
	\label{proposition:additional-properties-of-regularization}
	Suppose that $\cF \subseteq \extF{\cM}$ is a convex cone.
	Then for $\alpha_1, \alpha_2 \in \interval(){0}{\infty}$ and $f_1, f_2 \in \ExtF{\cM}$ with $\regularization{f_1} \not \equiv -\infty$ and $\regularization{f_2} \not \equiv -\infty$ we obtain
	\begin{equation}
		\label{eq:regularization_positive_linear_combination}
		\alpha_1 \regularization{f_1} + \alpha_2 \regularization{f_2}
		\le
		\regularization{\alpha_1 f_1 + \alpha_2 f_2}
		\le
		\alpha_1 f_1 + \alpha_2 f_2
		.
	\end{equation}
	For $\alpha \in \interval(){0}{\infty}$ and $f \in \ExtF{\cM}$ we obtain
	\begin{equation}\label{eq:regularization_positive_scaling}
		\alpha \regularization{f}
		=
		\regularization{\alpha f}
		.
	\end{equation}
\end{proposition}
\begin{proof}
	As we have observed just before \cref{proposition:elementary-properties-of-F-regularization}, $\regularization{f_i} \not \equiv -\infty$ implies $\regularization{f_i} \in \extF{\cM}$ and thus also $f_i \in \extF{\cM}$ for $i = 1, 2$.
	Thus, all functions in \eqref{eq:regularization_positive_linear_combination} are well-defined on all of $\cM$.

	Since the second inequality in \eqref{eq:regularization_positive_linear_combination} always holds due to \cref{proposition:elementary-properties-of-F-regularization}~\ref{item:elementary-properties-of-F-regularization:3}, let us show the first one.
	Fix $x \in \cM$, $\alpha_i \in \interval(){0}{\infty}$ and $M \in \R$ such that $\alpha_1 \regularization{f_1}(x) + \alpha_2 \regularization{f_2}(x) > M$.
	Then there exist $M_i \in \R$ such that $\regularization{f_i}(x) > M_i$ and $\alpha_1 \, M_1 + \alpha_2 \, M_2 \ge M$.
	Moreover, there exist $\varphi_i \in \affine{\cF}$ with the properties $\varphi_i \le f_i$ and $\varphi_i(x) \ge M_i$ by the definition of $\regularization{f_i}$ as the pointwise supremum over minorants in $\affine{\cF}$.
	We thus have
	\begin{equation*}
		\alpha_1 \, \varphi_1 + \alpha_2 \, \varphi_2
		\le
		\alpha_1 \, f_1 + \alpha_2 \, f_2
		.
	\end{equation*}
	Since $\cF$ is a convex cone, $\alpha_1 \, \varphi_1 + \alpha_2 \, \varphi_2 \in \affine{\cF}$.
	Thus,
	\begin{equation*}
		\regularization{\alpha_1 f_1 + \alpha_2 f_2}(x)
		\ge
		\alpha_1 \, \varphi_1(x) + \alpha_2 \, \varphi_2(x)
		\ge
		\alpha_1 M_1 + \alpha_2 M_2
		\ge
		M
		.
	\end{equation*}
	Since $M$ was arbitrary, 	\eqref{eq:regularization_positive_linear_combination} follows.

	Omitting the second term $\alpha_2 f_2$, an analogous but simpler computation shows that $\alpha \regularization{f} \le \regularization{\alpha f}$ holds in case $\regularization{f} \not \equiv -\infty$.
	In case $\regularization{f} \equiv -\infty$, the inequality $\alpha \regularization{f} \le \regularization{\alpha f}$ holds trivially.

	The converse inequality $\alpha \regularization{f} \ge \regularization{\alpha f}$ follows from choosing $\alpha_1 = \frac{1}{\alpha}$ and $f_1 = \alpha f$, which yields
	\begin{align*}
		\frac{1}{\alpha}\regularization{\alpha f}
		&
		=
		\alpha_1 \regularization{f_1}
		\le
		\regularization{\alpha_1 f_1}
		=
		\regularization{f}
		.
		\qedhere
	\end{align*}
\end{proof}

\subsection{The sup-Closure}
\label{subsection:sup-closure}

Let us now consider the set of all functions that can be produced by computing $\cF$-regularizations.
\begin{definition}
	The closure of $\cF \subseteq \ExtF{\cM}$ with respect to $\cF$-regularization is called the $\sup$-closure of $\cF$:
	\begin{equation}
		\label{eq:sup-closure:1}
		\supcl{\cF}
		\coloneqq
		\setDef[big]{\regularization{f}}{f \in \ExtF{\cM}}
		.
	\end{equation}
\end{definition}
Obviously, $-\infty = \regularization{-\infty} \in \supcl{\cF}$ holds for any $\cF \subseteq \ExtF{\cM}$.
\Cref{proposition:elementary-properties-of-F-regularization}~\ref{item:elementary-properties-of-F-regularization:2} yields that $g = \regularization{f} \in \supcl{\cF}$ implies $g + c = \regularization{f+c} \in \supcl{\cF}$ for any $c \in \R$, \ie,
\begin{equation}
	\label{eq:shift-invariance-of-the-sup-closure}
	\supcl{\cF}
	=
	\affine{\supcl{\cF}}
	.
\end{equation}
Thus, if $\cF$ contains an element of $\extF{\cM}$, then also $+\infty = \regularization{+\infty} \in \supcl{\cF}$.

Next, we give two alternative characterizations of the $\sup$-closure \eqref{eq:sup-closure:1}:
\begin{proposition}\label{pro:alternative_definition_to_supclosure}
	For any subset $\cF \subseteq \ExtF{\cM}$, $\supcl{\cF}$ contains precisely the pointwise suprema of subsets of $\affine{\cF}$, \ie,
	\begin{equation}
		\label{eq:sup-closure:2}
		\supcl{\cF}
		=
		\setDef[big]{\sup \setDef{\varphi}{\varphi \in \cE}}{\cE \subseteq \affine{\cF}}
		.
	\end{equation}
	$\supcl{\cF}$ is also the set of $f \in \ExtF{\cM}$ that coincide with their $\cF$-regularization,
	\begin{equation}
		\label{eq:sup-closure:3}
		\supcl{\cF}
		=
		\setDef[big]{f \in \ExtF{\cM}}{f = \regularization{f}}
		=
		\setDef[big]{f \in \ExtF{\cM}}{f \le \regularization{f}}
		.
	\end{equation}
\end{proposition}
\begin{proof}
	By definition, any $\regularization{f}$ is a pointwise supremum over a subset of $\affine{\cF}$, proving the inclusion~$\subseteq$ in \eqref{eq:sup-closure:2}.
	Conversely, if $f = \sup_{\varphi \in \cE}$ for some $\cE \subseteq \affine{\cF}$, then any $\varphi \in \cE$ satisfies $\varphi \le f$ and, clearly, $\varphi \in \affine{\cF}$.
	Thus $\cE \subseteq \setDef{\varphi \in \affine{\cF}}{\varphi \le f}$, which implies $f \le \regularization{f}$ by definition of $\regularization{f}$.
	Using \cref{proposition:elementary-properties-of-F-regularization}~\ref{item:elementary-properties-of-F-regularization:3}, we conclude $f = \regularization{f}$, \ie, $f \in \supcl{\cF}$.
	This concludes the proof of \eqref{eq:sup-closure:2}.

	Concerning the first equality of \eqref{eq:sup-closure:3}, we clearly have the implication $f = \regularization{f} \Rightarrow f \in \supcl{\cF}$.
	Conversely, $f \in \supcl{\cF}$ implies that $f = \regularization{g}$ for some $g \in \ExtF{\cM}$.
	Hence, by \cref{proposition:elementary-properties-of-F-regularization}~\ref{item:elementary-properties-of-F-regularization:5}, we compute $\regularization{f} = \regularization{\regularization{g}} = \regularization{g} = f$.
	The second equality of \eqref{eq:sup-closure:3} follows from \cref{proposition:elementary-properties-of-F-regularization}~\ref {item:elementary-properties-of-F-regularization:3}.
\end{proof}
\begin{corollary}
	\label{corollary:closeness-of-sup-closure-under-suprema-and-limes-inferiors}
	Suppose that $\Lambda$ is any index set, $f \in \ExtF{\cM}$ and $\cF \subseteq \ExtF{\cM}$.
	\begin{enumerate}
		\item \label[statement]{item:closeness-of-sup-closure-under-suprema-and-limes-inferiors:1}
			Suppose $f_\lambda \in \supcl{\cF}$ for all $\lambda \in \Lambda$, and $f \coloneqq \sup_{\lambda \in \Lambda} f_\lambda$.
			Then $f \in \supcl{\cF}$.

		\item \label[statement]{item:closeness-of-sup-closure-under-suprema-and-limes-inferiors:2}
			Suppose $f_k \in \supcl{\cF}$ for all $k \in \N$ and $f \coloneqq \liminf_{k \to \infty} f_k$.
			If $\varepsilon_k$ is a non-negative real-valued sequence such that $\liminf_{k \to \infty} \varepsilon_k = 0$ and $f_k \le f + \varepsilon_k$, then $f \in \supcl{\cF}$.
	\end{enumerate}
\end{corollary}
\begin{proof}
	We show $\regularization{f} \ge f$ which holds, if and only if $f \in \supcl{\cF}$ by \eqref{eq:sup-closure:3}.

	For \cref{item:closeness-of-sup-closure-under-suprema-and-limes-inferiors:1}, we use the monotonicity (\cref{proposition:elementary-properties-of-F-regularization}~\ref{item:elementary-properties:1}) of $\regularization{\cdot}$ to compute
	\begin{equation*}
		f
		=
		\sup_{\lambda \in \Lambda} f_\lambda
		=
		\sup_{\lambda \in \Lambda} \regularization{f_\lambda}
		\le
		\regularization{\sup_{\lambda \in \Lambda} f_\lambda}
		=
		\regularization{f}
		.
	\end{equation*}
	Concerning \cref{item:closeness-of-sup-closure-under-suprema-and-limes-inferiors:2}, we consider
	\begin{equation*}
		f_k
		=
		\regularization{f_k}
		\le
		\regularization{f+\varepsilon_k}
		=
		\regularization{f}+\varepsilon_k
		.
	\end{equation*}
	Taking the $\liminf_{k \to \infty}$ on both sides, we obtain $f \le \regularization{f}$ and thus the desired result.
\end{proof}

The following result strengthens \cref{proposition:additional-properties-of-regularization}.
\begin{proposition}
	Suppose that $\cF \subseteq \extF{\cM}$ is a convex cone.
	Then $\supcl{\cF} \setminus \set{-\infty} \subseteq \extF{\cM}$ is a convex cone as well, and for $f_1, f_2 \in \supcl{\cF} \setminus \set{-\infty}$ and $\alpha_1, \alpha_2 \in \interval(){0}{\infty}$, we have
	\begin{equation}
		\label{eq:equality_of_regularization_in_supcl}
		\regularization{\alpha_1 \, f_1 + \alpha_2 \, f_2}
		=
		\alpha_1 \,f_1
		+
		\alpha_2 \,f_2
		.
	\end{equation}
\end{proposition}
\begin{proof}
  As we have seen just before \cref{proposition:elementary-properties-of-F-regularization}, $\cF \subseteq \extF{\cM}$ implies that $\supcl{\cF}\subseteq \extF{\cM}\cup \set{-\infty}$.
	Hence $\supcl{\cF} \setminus \set{-\infty} \subseteq \extF{\cM}$.
	If $f_i \in \supcl{\cF} \setminus \set{-\infty}$ for $i=1, 2$, then $\regularization{f_i} = f_i \not \equiv -\infty$ by \eqref{eq:sup-closure:3}, and \eqref{eq:regularization_positive_linear_combination} can be applied.
	Hence \eqref{eq:equality_of_regularization_in_supcl} follows, which in turn implies $\alpha_1 \, f_1 + \alpha_2 \,f_2 \in \supcl{\cF} \setminus \set{-\infty}$ by \eqref{eq:sup-closure:3}.
	Thus, $\supcl{\cF} \setminus \set{-\infty}$ is a convex cone.
\end{proof}

Next we prove that the $\sup$-closure, as the name suggests, is indeed a closure (\aka hull) operator; see, \eg, \cite{Westermann:1976:1}.
\begin{proposition}
	\label{proposition:sup-closure}
	Suppose $\cF, \cG \subseteq \ExtF{\cM}$.
	Then the following hold:
	\begin{subequations}
		\label{eq:sup-closure-is-a-hull-operator}
		\begin{align}
			&
			\cF
			\subseteq
			\supcl{\cF}
			\label{eq:sup-closure-is-a-hull-operator:1}
			\\
			&
			\cF
			\subseteq
			\cG
			\quad
			\Rightarrow
			\quad
			\supcl{\cF}
			\subseteq
			\supcl{\cG}
			\label{eq:sup-closure-is-a-hull-operator:2}
			\\
			&
			\supcl{\supcl{\cF}}
			=
			\supcl{\cF}
			.
			\label{eq:sup-closure-is-a-hull-operator:3}
		\end{align}
	\end{subequations}
	That is, $\supcl{\cdot}$ is a closure operator on $\ExtF{\cM}$.
\end{proposition}
\begin{proof}
	By \cref{proposition:elementary-properties-of-F-regularization}~\ref{item:elementary-properties-of-F-regularization:4}, $f \in \cF$ implies $f = \regularization{f} \in \supcl{\cF}$.
	This shows \eqref{eq:sup-closure-is-a-hull-operator:1}.
	\Cref{eq:sup-closure-is-a-hull-operator:2} follows directly from \cref{pro:alternative_definition_to_supclosure}, since any subset of $\affine{\cF}$ is a subset of $\affine{\cG}$.
	Finally, by \eqref{eq:sup-closure-is-a-hull-operator:1} and \eqref{eq:sup-closure-is-a-hull-operator:2}, we have the inclusion $\supcl{\cF} \subseteq \supcl{\supcl{\cF}}$.
	Conversely, suppose now $f \in \supcl{\supcl{\cF}}$.
	Using \eqref{eq:sup-closure:2}, we can write $f = \sup_{\varphi \in \cE} \varphi$ for some $\cE \subseteq \supcl{\cF}$.
	By \cref{corollary:closeness-of-sup-closure-under-suprema-and-limes-inferiors}~\ref{item:closeness-of-sup-closure-under-suprema-and-limes-inferiors:1}, $\supcl{\cdot}$ is closed under taking suprema, so this implies $f \in \supcl{\cF}$, and thus we have $\supcl{\cF} \supseteq \supcl{\supcl{\cF}}$.
\end{proof}

It follows from general properties of closure operators that the $\sup$-closure of a set can be obtained by intersecting all $\sup$-closed supersets, \ie,
\begin{equation*}
	\supcl{\cF}
	=
	\bigcap \setDef[big]{\cG \subseteq \ExtF{\cM}}{\cG = \supcl{\cG} \text{ and } \cF \subseteq \cG}
	.
\end{equation*}

The following result shows that the value of $\genconj{g}(f)$ for $f \in \supcl{\cF} \subseteq \ExtF{\cM}$ is uniquely determined by its values on $\cF$:
\begin{proposition}
	Suppose that $\cF \subseteq \ExtF{\cM}$ and $f, g \in \ExtF{\cM}$.
	Then
	\begin{equation}
		\label{eq:supoversubspace}
		\genconj{g}(\regularization{f})
		=
		\sup \setDef[big]{\genconj{g}(\varphi)}{\varphi \in \affine{\cF}, \; \varphi \le f}
	\end{equation}
	holds.
\end{proposition}
\begin{proof}
	The result follows from the definition of $\regularization{f}$ and \cref{proposition:continuity-of-generalized-conjugation}~\ref{item:continuity-of-generalized-conjugation:1}:
	\begin{align*}
		\genconj{g}(\regularization{f})
		&
		=
		\genconj{g} \paren[big](){\sup \setDef[big]{\varphi}{\varphi \in \affine{\cF}, \; \varphi \le f}}
		=
		\sup \setDef[big]{\genconj{g}(\varphi)}{\varphi \in \affine{\cF}, \; \varphi \le f}
		.
		\qedhere
	\end{align*}
\end{proof}
\begin{proposition}\label{proposition:Fenchel_of_Regularization}
	Suppose that $\cF \subseteq \extF{\cM}$ and $f \in \ExtF{\cM}$.
	Then
	\begin{equation}
		\label{eq:Fenchelreg}
		\genconj{(\regularization{f})}(\varphi)
		=
		\genconj{f}(\varphi)
		\quad
		\text{for all }
		\varphi \in \supcl{\cF}
		.
	\end{equation}
\end{proposition}
\begin{proof}
	In general, since $\regularization{f}\le f$ we have $\genconj{(\regularization{f})}(\varphi)\ge \genconj{f}(\varphi)$ for all $\varphi \in \ExtF{\cM}$. In particular, equality holds for $\genconj{f}(\varphi)=\infty$.
	Since the case $\genconj{f}(\varphi) = -\infty \Leftrightarrow (\varphi \equiv -\infty$ or $f \equiv +\infty)$ is trivial, we may assume $\genconj{f}(\varphi)\in \R$.
	By the Fenchel-Young inequality \eqref{eq:FYI:3} we know
	\begin{equation*}
		\varphi(x)
		\le
		f(x) + \genconj{f}(\varphi)
		\quad
		\text{for all }
		x \in \cM
		.
	\end{equation*}
	Thus, using $\varphi \in \supcl{\cF}$ and \Cref{proposition:elementary-properties-of-F-regularization} \ref{item:elementary-properties-of-F-regularization:1} and \ref{item:elementary-properties-of-F-regularization:2}:
	\begin{equation*}
		\varphi
		=
		\regularization{\varphi}
		\le
		\regularization{f+\genconj{f}(\varphi)}
		=
		\regularization{f}+\genconj{f}(\varphi)
		.
	\end{equation*}
	Hence, $\genconj{f}(\varphi) \ge \varphi(x)-\regularization{f}(x)$ for all $x \in \cM$, and thus $\genconj{f}(\varphi) \ge \genconj{(\regularization{f})}(\varphi)$.
\end{proof}

\begin{example}
	\label{example:sup-closures}
	Important examples for the choice of $\cF$ are the following:
	\begin{enumerate}
		\item
			Suppose that $\cM$ is a locally convex linear topological space and $\cF = \dualSpace{\cM}$ is its topological dual space.
			Then $\affine{\cF}$ is the space of all continuous affine functions and $\regularization[\dualSpace{\cM}]{f}$ is the pointwise supremum over all affine majorants of $f$.
			It is known that $\supcl{\cF}$ is the set of lower semi-continuous convex functions on $\cM$; see, \eg, \cite[Proposition~I.3.2]{EkelandTemam:1999:1}.

		\item
			Suppose that $\cM$ is a metric space.
			Then by a theorem of Baire, lower semi-continous functions $f \in \extF{\cM}$ can be written as the pointwise supremum of continuous functions; see, \eg, \cite[Theorem~16.16]{Schechter:1997:1} for the more general setting of a completely regular topological space.
			Thus, for $\cF = C(\cM)$, $\supcl{\cF}$ consists of the cone of lower semi-continuous functions in $\extF{\cM}$ (which is closed under taking pointwise suprema) and $\set{-\infty}$.

		\item
			In the following section we will see that a similar relation holds (under mild assumptions), in case $\cM$ is a $C^1$-Banach manifold and $\cF = C^1(\cM)$.
	\end{enumerate}
\end{example}

\begin{remark}
	\label{remark:relation-to-work-of-MartinezLegaz}
	In his work on generalized convexity, \cite{MartinezLegaz:2005:1} considers an alternative generalized notion of the Fenchel conjugate.
	He considers two non-empty sets $\cM, \cN$ and a coupling function $c \colon \cM \times \cN \to \ExtR$.
	His definition of $c$-conjugate is then
	\begin{equation*}
		f^c(y)
		\coloneqq
		\sup_{x \in \cM} c(x,y) - f(x)
		\quad
		\text{for }
		y
		\in
		\cN
		.
	\end{equation*}
	This can be related to our approach of nonlinear Fenchel conjugates \eqref{eq:general-nonlinear-Fenchel-conjugate} as follows:
	\begin{enumerate}
		\item
			Given $c$ and $\cN$, we can define $\varphi_y \coloneqq c(\cdot,y)$ to obtain $\genconj f(\varphi_y) = f^c(y)$ for all $y \in \cN$.
			Choosing $\cF \coloneqq \setDef{\varphi_y}{y \in \cN}$ yields $\restr{\genconj{f}}{\cF} \equiv f^c$.

		\item
			Vice versa, given $\cF \subseteq \ExtF{\cM}$, we may choose $\cN \coloneqq \cF$ and $c(x,\varphi) \coloneqq \varphi(x)$.
			Then again, we have $\genconj f(\varphi) = f^c(\varphi)$ for all $\varphi \in \cF$.
	\end{enumerate}
	That is, \cite{MartinezLegaz:2005:1} generalizes the duality pairing while we generalize the set of test functions.
\end{remark}

\begin{proposition}
	\label{proposition:pullback}
	Suppose that $A \colon \cM \to \cN$ is some mapping, $\cF \subseteq \ExtF{\cN}$ and denote
	\begin{equation*}
		\genadj{A}(\cF)
		\coloneqq
		\setDef{\genadj{A}(\varphi)}{\varphi \in \cF}
		=
		\setDef{\varphi \circ A}{\varphi \in \cF} \subseteq \ExtF{\cM}
		.
	\end{equation*}
	Then the following holds:
	\begin{enumerate}
		\item \label[statement]{item:pullback:1}
			$\genadj{A}(\regularization{f}) \le \regularization[\genadj{A}(\cF)]{\genadj{A}(f)}$.
			In case $A$ is surjective, we have equality.

		\item \label[statement]{item:pullback:2}
			$\genadj{A}(\supcl{\cF}) = \supcl{\genadj{A}(\cF)}$.
	\end{enumerate}
\end{proposition}
\begin{proof}
	\Cref{item:pullback:1} follows from the definitions by a straightforward calculation:
	\begin{align*}
		\paren[big](){\genadj{A}(\regularization{f})}(x)
		&
		=
		\paren[big](){\regularization{f} \circ A}(x)
		\\
		&
		=
		\paren[big](){\sup \setDef{\varphi}{\varphi \in \affine{\cF}, \; \varphi \le f} \circ A}(x)
		\\
		&
		=
		\sup \setDef[big]{(\varphi \circ A)(x)}{\varphi \in \affine{\cF}, \; \varphi \le f}
		\\
		&
		\le
		\sup \setDef[big]{(\varphi \circ A)(x)}{\varphi \in \affine{\cF}, \; \varphi \circ A \le f \circ A}
		\\
		&
		=
		\sup \setDef[big]{\psi(x)}{\psi \in \genadj{A}(\affine{\cF}), \; \psi \le \genadj{A}(f)}
		\\
		&
		=
		\sup \setDef[big]{\psi(x)}{\psi \in \affine{\genadj{A}(\cF)}, \; \psi \le \genadj{A}(f)}
		=
		\regularization[\genadj{A}(\cF)]{\genadj{A}(f)}(x)
		.
	\end{align*}%
	The inequality becomes an equality if and only if $\varphi \le f \Leftrightarrow \varphi \circ A \le f \circ A$, which is true in case $A$ is surjective.

	We now turn to \cref{item:pullback:2}.
	Suppose that $g \in \genadj{A}(\supcl{\cF})$, \ie, $g = \genadj{A}(f)$ for some $f \in \supcl{\cF}$.
	The latter condition is equivalent to $f = \regularization{f}$ by \eqref{eq:sup-closure:3}.
	We thus obtain $g = \genadj{A}(\regularization{f})$.
	By \cref{item:pullback:1}, we have $g = \genadj{A}(\regularization{f}) \le \regularization[\genadj{A}(\cF)]{\genadj{A}(f)} = \regularization[\genadj{A}(\cF)]{g}$.
	Thus by \cref{proposition:elementary-properties-of-F-regularization}~\ref{item:elementary-properties:3} $g=\regularization[\genadj{A}(\cF)]{g}$ and
	we conclude $g \in \supcl{\genadj{A}(\cF)}$, \ie, $\genadj{A}(\supcl{\cF}) \subseteq \supcl{\genadj{A}(\cF)}$.

	For the reverse inclusion, consider some function $g \in \supcl{\genadj{A}(\cF)}$, \ie,
	\begin{equation*}
		\begin{aligned}
			g
			=
			\regularization[\genadj{A}(\cF)]{g}
			&
			=
			\sup \setDef[big]{\varphi}{\varphi \in \affine{\genadj{A}(\cF)}, \; \varphi \le g}
			\\
			&
			=
			\sup \setDef[big]{\psi \circ A}{\psi \in \affine{\cF}, \; \psi \circ A \le g}
			.
		\end{aligned}
	\end{equation*}
	The pointwise evaluation of $g$ thus yields
	\begin{equation*}
		g(x)
		=
		\sup \setDef[big]{\psi(A(x))}{\psi \in \affine{\cF}, \; \psi \circ A \le g}
		\text{ for }
		x
		\in
		\cM
		.
	\end{equation*}
	We now define
	\begin{equation*}
		f(y)
		\coloneqq
		\sup \setDef[big]{\psi(y)}{\psi \in \affine{\cF}, \; \psi \circ A \le g}
		\text{ for }
		y
		\in
		\cN
	\end{equation*}
	and observe that $f \in \supcl{\cF}$ by \eqref{eq:sup-closure:2}.
	Since $g(x) = f(A(x))$ holds for all $x \in \cM$, we have $g = f \circ A = \genadj{A}(f)$.
	This shows $g \in \genadj{A}(\supcl{\cF})$ and thus $\supcl{\genadj{A}(\cF)} \subseteq \genadj{A}(\supcl{\cF})$.
\end{proof}

\subsection{\texorpdfstring{$\cF$}{𝓕}-Biconjugates}
\label{subsection:F-biconjugation}

In this section we define an appropriate notion of biconjugation that generalizes the corresponding classical concept from the linear case.
In the linear setting, the biconjugate~$\biconjugate{f} \in \ExtF{V}$ conincides with the largest convex \lsc minorant of~$f \in \ExtF{V}$.
As we will see, this result carries over to the nonlinear case.

To define the biconjugate, consider the restriction of $\genconj{f} \in \ExtF{\cM}$ to~$\cF$, \ie,
\begin{equation*}
	\restr{\genconj{f}}{\cF}
	\colon
	\cF
	\to \ExtR
\end{equation*}
and take the conjugate of this mapping:
\begin{equation*}
	\genconj{(\restr{\genconj{f}}{\cF})}
	\colon
	\ExtF{\cF} \to \ExtR
	.
\end{equation*}
It is clear that all results up to now can be applied to $\genconj{(\restr{\genconj{f}}{\cF})}$ with $\ExtF{\cF}$ in place of~$\cM$.
However, $\ExtF{\cF}$ can be a very large domain of definition.
To be able to compare $\genconj{(\restr{\genconj{f}}{\cF})}$ with~$f$, we have to restrict the domain again.
To this end, we consider evaluation (Dirac) functions
\begin{equation*}
	\begin{aligned}
		\delta_x
		\colon
		\ExtF{\cM}
		&
		\to
		\ExtR
		\\
		\varphi
		&
		\mapsto
		\delta_x(\varphi)
		\coloneqq
		\varphi(x)
		.
	\end{aligned}
\end{equation*}
Clearly, the restriction of $\delta_x$ to a linear subspace $\cF \subseteq \realF(\cM)$ is a linear function on~$\cF$ and continuous with respect to pointwise convergence.
We have the nonlinear canonical embedding of~$\cM$ into the algebraic dual space of~$\cF$ via
\begin{equation}
	\label{eq:canemb}
	\begin{aligned}
		\bidualembedding{\cM}{\algebraicdualSpace{\cF}}
		\colon
		\cM
		&
		\to
		\algebraicdualSpace{\cF}
		\\
		x
		&
		\mapsto \delta_x
		.
	\end{aligned}
\end{equation}

\begin{definition}
	\label{definition:nonlinear-biconjugate}
	Suppose that $\cF$ is a linear subspace of~$\realF(\cM)$.
	Then we define the $\cF$-biconjugate $\bigenconj{f}{\cF}$ of $f \in \ExtF{\cM}$ as $\bigenconj{f}{\cF} \coloneqq \genconj{(\restr{\genconj{f}}{\cF})} \circ \bidualembedding{\cM}{\algebraicdualSpace{\cF}}$, \ie,
	\begin{equation}
		\label{eq:nonlinear-biconjugate}
		\begin{aligned}
			\bigenconj{f}{\cF}
			\colon
			\cM
			&
			\to
			\ExtR
			\\
			x
			&
			\mapsto
			\genconj{(\restr{\genconj{f}}{\cF})}(\delta_x)
			.
		\end{aligned}
	\end{equation}
\end{definition}
When written in the form~\eqref{eq:general-nonlinear-Fenchel-conjugate:alternative:3}, this definition reads
\begin{equation}
	\label{eq:bicon2}
	\bigenconj{f}{\cF}(x)
	=
	\sup \setDef[big]{\varphi(x) - c}{(\varphi,c) \in \epi \restr{\genconj{f}}{\cF}}
	.
\end{equation}

\begin{theorem}
	\label{theorem:biconjugate}
	Suppose that $\cF$ is a linear subspace of~$\realF(\cM)$.
	The $\cF$-biconjugate satisfies $\bigenconj{f}{\cF} = \regularization{f}$ for all $f \in \ExtF{\cM}$.
\end{theorem}
\begin{proof}
	We write the equivalent form \eqref{eq:bicon2} of the biconjugate~$\bigenconj{f}{\cF}$ in a different way:
	\begin{equation}
		\label{eq:bicon3}
		\begin{aligned}
			\bigenconj{f}{\cF}(x)
			&
			=
			\sup \setDef[big]{\varphi(x) - c}{\varphi \in \cF, \; c \in \R, \; \genconj{f}(\varphi) - c \le 0}
			\\
			&
			=
			\sup \setDef[big]{\varphi(x)}{\varphi \in \affine{\cF}, \; \genconj{f}(\varphi) \le 0}
			.
		\end{aligned}
	\end{equation}
	The second equality uses \eqref{eq:FDconstant}.
	We thus observe that $\bigenconj{f}{\cF}$ is a pointwise supremum of functions $\varphi \in \affine{\cF}$.
	By~\eqref{eq:elementary-properties:3}, we have $\varphi \le f \Leftrightarrow \genconj{f}(\varphi) \le 0$.
	Hence the characterizations of $\regularization{f}$ via~\eqref{eq:defreg} and $\bigenconj{f}{\cF}$ via~\eqref{eq:bicon3} coincide.
\end{proof}

We now investigate the triple $\cF$-conjugate.
\begin{corollary}
	\label{corollary:triconjugate}
	Suppose that $\cF$ is a linear subspace of~$\realF(\cM)$.
	Then $\genconj{(\bigenconj{f}{\cF})}(\varphi) = \genconj{f}(\varphi)$ holds for any $\varphi \in \supcl{\cF}$.
\end{corollary}
\begin{proof}
	This is a direct consequence of \Cref{theorem:biconjugate} and \Cref{proposition:Fenchel_of_Regularization}.
\end{proof}

\begin{remark}
	Our results generalize the notion of the classical biconjugate on locally convex or normed linear spaces~$V$.
	In that setting, the classical biconjugate $\biconjugate{f} \colon V \to \ExtR$ is the composition of $\conjugate{(\conjugate{f})} \colon \bidualSpace{V} \to \ExtR$ with the linear canonical embedding, in short, $\biconjugate{f} \coloneqq \conjugate{(\conjugate{f})} \circ \bidualembedding{V}{\bidualSpace{V}}$.
	In other words, $\biconjugate{f} = \bigenconj{f}{\dualSpace{V}}$ holds according to \cref{definition:nonlinear-biconjugate}.
\end{remark}

\section{Nonlinear Fenchel Conjugates on Manifolds}
\label{section:nonlinear-Fenchel-conjugates:manifolds}

So far we have considered the nonlinear Fenchel conjugate $\genconj{f}$ to act on the very general class of arguments $\varphi \in \ExtF{\cM}$ over general non-empty sets~$\cM$.
We saw that elementary properties of the classical Fenchel conjugate generalize to this setting.
To obtain further insight, in particular into the connection between Fenchel conjugates and subdifferentials, however, we have to restrict the arguments~$\varphi$ to a suitable subspace of smooth functions.
Thus, a natural setting for this analysis involves sufficiently smooth manifolds.

In this section we consider a $C^k$-Banach manifold~$\cM$ for some $k \in \N_0 \cup \set{+\infty}$, modeled on a real Banach space~$X$, and corresponding spaces of smooth functions $\cF = C^k(\cM)$.
In case $\cM$ admits $C^k$-partitions of unity (\cref{subsection:smooth-partitions-of-unity}), it will turn out that $\supcl{C^k(\cM)}$ consists of all lower semi-continuous functions on $\cM$; see \cref{subsection:nonlinear-Fenchel-conjugates:manifolds:smooth-functions}.

Moreover, for $k \ge 1$, we can prove a version of the well known Fenchel-Young theorem, connecting the viscosity Fréchet subdifferential with the nonlinear Fenchel conjugate (\cref{subsection:nonlinear-Fenchel-conjugates:manifolds:conjugates-and-subdifferentials}).
Since a real Banach space is also $C^\infty$-Banach manifold, this special case is naturally included throughout this section.
\Cref{subsection:nonlinear-Fenchel-conjugates:manifolds:comparison} provides a comparison to previous definitions of Fenchel conjugates on manifolds.

\subsection{Smooth Partitions of Unity}%
\label{subsection:smooth-partitions-of-unity}

The existence of smooth partitions of unity is a fundamental assumption on manifolds that is usually required to globalize local results.
We will need this property in a couple of instances and recall its definition for convenience of the reader.

\begin{definition}
	\label{definition:partition-of-unity}
	Suppose that $k \in \N_0 \cup \set{+\infty}$ and $\cM$ is a $C^k$-Banach manifold.
	\begin{enumerate}
		\item
			Suppose that $\cC \coloneqq \setDef[big]{\cO_\lambda}{\lambda \in \Lambda}$ is an open cover of~$\cM$.
			A family of $C^k$-functions $\varphi_\lambda \in C^k(\cM,[0,1])$, $\lambda \in \Lambda$, is said to be a locally finite $C^k$-partition of unity of~$\cM$ subordinate to~$\cC$ if
			\begin{itemize}
				\item
					$\supp \varphi_\lambda \subseteq \cO_\lambda$ for all $\lambda \in \Lambda$,

				\item
					every $x \in \cM$ has a neighborhood~$\cU$ such that $\cU \cap \supp \varphi_\lambda = \emptyset$ for all but a finite number of $\lambda \in \Lambda$,

				\item
					the finite sum $\sum_{\lambda \in \Lambda} \varphi_\lambda(x) = 1$ for all $x \in \cM$.
			\end{itemize}

		\item
			$\cM$ is said to admit locally finite $C^k$-partitions of unity if for any open cover $\cC \coloneqq \setDef{\cO_\lambda}{\lambda \in \Lambda}$ of~$\cM$, there exists a locally finite $C^k$-partition of unity subordinate to $\cC$.
	\end{enumerate}
\end{definition}
In what follows, we will simply speak of $C^k$-partitions of unity without explicitly mentioning local finiteness.

Continuous partitions of unity ($k = 0$) exist in paracompact topological spaces, in particular in metrizable spaces, and are thus purely a topic of topology.
The existence of $C^k$-partitions of unity on Banach manifolds has been studied in \cite{BonicFrampton:1966:1}.
It has been shown that a $C^k$-Banach manifold~$\cM$ admits $C^k$-partitions of unity if $\cM$ is separable and is modeled on a Banach space~$X$ that has an equivalent norm of class~$C^k$ on $X \setminus \set{0}$.
In this case, it is known that $C^k(\cM)$ is dense in $C(\cM)$.
Examples for Banach spaces~$X$ of this class for $k = 1$ include separable Banach spaces with separable dual spaces; see \cite[Proposition~3]{BonicFrampton:1966:1}.
This class comprises, in particular, separable reflexive Banach spaces.
For $k = 2$, suitable Banach spaces~$X$ include separable Hilbert spaces.

A function $f \in \ExtF{\cM}$ is said to be lower semi-continuous at $x \in \cM$ if, for each $\varepsilon > 0$, there is a neighborhood~$\cU$ of~$x$ such that $f(y) \ge f(x) - \varepsilon$ for all $y \in \cU$.
Let us denote by
\begin{equation*}
	\lscextF{\cM}
	\coloneqq
	\setDef{f \colon \cM \to \extR}{f \text{ is lower semi-continuous at every } x \in \cM}
	.
\end{equation*}

\begin{lemma}[A $C^k$-Sandwich Lemma]
	\label{lemma:Ck-sandwich-lemma}
	Suppose that $k \in \N_0 \cup \set{\infty}$ and $\cM$ is a $C^k$-Banach manifold that admits $C^k$-partitions of unity.
	Suppose that $\cO \subseteq \cM$ is some open set, $w \in C^k(\cO)$ and $-a, b \in \lscextF{\cM}$ such that $a \le w \le b$ on~$\cO$ and $a < b$ on $\cM \setminus \cO$.
	Then, given a closed set $\cA \subseteq \cO$, there exists a function $h \in C^k(\cM)$ satisfying $a \le h \le b$ on~$\cM$, $a < h < b$ on $\cM \setminus \cO$, and $h = w$ on~$\cA$.
\end{lemma}
\begin{proof}
	The result is a slight extension of Dowker's sandwich theorem, which was originally formulated for $k = 0$ and $\cO = \emptyset$; see, \eg, \cite[Theorem~VIII.4.3]{Dugundji:1978:1}.

	Given $r \in \R$, consider the open sets $\cG_r \coloneqq \setDef{x \in \cM \setminus \cA}{a(x) < r < b(x)}$.
	The openness of $\cG_r$ follows from the closedness of $\cA$ and the lower semi-continuity of $-a$ and $b$.
	Therefore, $\setDef{\cG_r}{r \in \R} \cup \set{\cO}$ is an open cover of~$\cM$, since $\cA \subseteq \cO$ and $a < b$ holds on $\cM \setminus \cO$.
	(We could make this cover countable by restricting to $r \in \Q$.)

	Suppose that $\setDef{g_r}{r \in \R} \cup \set{g_\cO}$ is a $C^k$-partition of unity, subordinate to this open cover.
	Denote by $w \, g_\cO$ the extension by zero from $\cO$ to $\cM$ of the function $x \mapsto w(x) \, g_\cO(x)$.
	This function is of class~$C^k$ on~$\cO$ by the product rule, and for $x \in \cM \setminus \cO \subseteq \cM \setminus \supp g_\cO$, there is a neighborhood of~$x$ on which $w \, g_\cO$ vanishes.
	Thus $w \, g_\cO \in C^k(\cM)$.

	Define the function $h \in C^k(\cM)$ via the (locally finite) sum:
	\begin{equation*}
		h(x)
		\coloneqq
		\sum_{r \in \R} r \, g_r(x) + w \, g_\cO(x)
		.
	\end{equation*}
	Due to $\cA \cap \cG_r = \emptyset$ and thus $g_r \equiv 0$ on~$\cA$ for all $r \in \R$, we have $h = w$ on~$\cA$.
	Moreover, since $\sum_{r \in \R} g_r + g_\cO \equiv 1$, we observe $h \le \paren[auto](){\sum_{r \in \R} g_r + g_\cO} \, b = b$ and $h \ge \paren[auto](){\sum_{r \in \R} g_r + g_\cO} \, a = a$.
	In both estimates, strict inequality holds if $\sum_{r \in \R} g_r(x) \neq 0$.
	This is the case, in particular, for $x \not \in \cO$.
\end{proof}

\begin{corollary}
	\label{corollary:Ck-sandwich-lemma}
	Suppose that $k \in \N_0 \cup \set{\infty}$ and $\cM$ is a $C^k$-Banach manifold that admits $C^k$-partitions of unity.
	Assume that $\cO \subseteq \cM$ is some open set and $f \in \lscextF{\cM}\cap C^k(\cO)$.
	Then, given a closed set $\cA \subseteq \cO$, $x_0 \in \cM$ and $\varepsilon > 0$, there is a function $\varphi \in C^k(\cM)$ satisfying $\varphi \le f$ on $\cM$, $\varphi = f$ on $\cA$ and $\varphi(x_0) > f(x_0) - \varepsilon$.
\end{corollary}
\begin{proof}
	We apply \cref{lemma:Ck-sandwich-lemma} with the choices $b \coloneqq f$, $a \coloneqq - \indicator{\set{x_0}} + f(x_0) - \varepsilon$ and $w \coloneqq \restr{f}{\cO}$ to obtain $\varphi \coloneqq h$ with the desired properties.
\end{proof}

\subsection{Restriction to Smooth Test Functions}%
\label{subsection:nonlinear-Fenchel-conjugates:manifolds:smooth-functions}

As announced in the beginning of \cref{section:nonlinear-Fenchel-conjugates:manifolds}, we now restrict the Fenchel conjugate $\genconj{f}$ to the domain
\begin{equation*}
	\cF
	\coloneqq
	C^k(\cM)
	\subseteq
	\realF(\cM)
	,
\end{equation*}
the linear space of $k$-times continuously differentiable real-valued functions on the $C^k$-Banach manifold~$\cM$.
Compared to the classical setting where we have a linear space~$V$ instead of~$\cM$, here the space $C^k(\cM)$ of smooth functions plays the role of the topological dual space~$\dualSpace{V}$ of continuous linear functions.

We begin with a result on the $\sup$-closure of $C^k$-functions.
\begin{proposition}
	\label{proposition:sup-closure-of-Ck-functions}
	Suppose that $k \in \N_0 \cup \set{\infty}$ and $\cM$ is a $C^k$-Banach manifold that admits $C^k$-partitions of unity.
	Then
	\begin{equation}
		\label{eq:sup-closure-of-Ck-functions}
		\supcl{C^k(\cM)}
		=
		\lscextF{\cM} \cup \set{-\infty}
		.
	\end{equation}
	Moreover, we have
	\begin{equation*}
		\genconj{f}(\psi)
		=
		\sup \setDef[big]{\genconj{f}(\varphi)}{\varphi \in C^k(\cM), \; \varphi \le \psi}
		\quad
		\text{for }
		\psi
		\in
		\lscextF{\cM}
		.
	\end{equation*}
	Hence, $\genconj{f}(\psi)$ is uniquely determined by $\restr{\genconj{f}}{C^k(\cM)}$.
\end{proposition}
\begin{proof}
	By \eqref{eq:sup-closure:2}, $f \in \supcl{C^k(\cM)}$ is the pointwise supremum of some subset $\cE \subseteq C^k(\cM)$ and thus lower semi-continuous.
	If $\cE = \emptyset$, we have $f \equiv -\infty$.
	In all other cases, we have $f(x) > -\infty$ for all $x \in \cM$.
	Hence $\supcl{C^k(\cM)} \subseteq \lscextF{\cM} \cup \set{-\infty}$.
	For the reverse inclusion, suppose that $f \in \lscextF{\cM}$.
	\Cref{corollary:Ck-sandwich-lemma} with $\cO = \emptyset$ and arbitrary $x_0 \in \cM$ implies
	\begin{equation*}
		f(x_0)
		=
		\sup \setDef[big]{\varphi(x_0)}{\varphi \in C^k(\cM), \; \varphi \le f}
		.
	\end{equation*}
	This in turn yields $f = \regularization[C^k(\cM)]{f}$, hence $f \in \supcl{C^k(\cM)}$ by \eqref{eq:sup-closure:3}.
	Since we also have $-\infty \in \supcl{C^k(\cM)}$, \eqref{eq:sup-closure-of-Ck-functions} is proved.
	The second assertion now follows from~\eqref{eq:supoversubspace}.
\end{proof}

In the evaluation of nonlinear Fenchel conjugates $\genconj{f}(\varphi)$, the test function $\varphi$ may sometimes be given only locally on some open set~$\cV \subseteq \cM$.
In this case we only have access to $\sup_{x \in \cV} \paren\{\}{\varphi(x) - f(x)}$.
This term can be interpreted as the Fenchel conjugate of $f$ modified by the indicator function of~$\cV$, \ie,
\begin{equation}
	\label{eq:local-Fenchel-conjugate:1}
	\genconj{(f + \indicator{\cV})}(\varphi)
	=
	\sup_{x \in \cV} \paren[big]\{\}{\varphi(x) - f(x)}
	.
\end{equation}
The following result shows that such local Fenchel conjugates can be extended to global ones without significant changes in value:
\begin{proposition}
	\label{proposition:extension-of-local-Fenchel-conjugates}
	Suppose that $k \in \N_0 \cup \set{\infty}$ and $\cM$ is a $C^k$-Banach manifold that admits $C^k$-partitions of unity.
	Consider $\cA \subseteq \cV \subseteq \cM$ with $\cA$ closed and $\cV$ open.
	Moreover, suppose $f \in \lscextF{\cM}$, $\varphi \in \lscextF{\overline \cV} \cap C^k(\cV)$ and $\genconj{f}(\varphi) < +\infty$.
	Then, for any $\varepsilon > 0$, there exists $\psi \in C^k(\cM)$ such that $\psi = \varphi$ on $\cA$ and
	\begin{equation*}
		\genconj{(f + \indicator{\cV})}(\varphi)
		-
		\varepsilon
		\le
		\genconj f(\psi)
		\le
		\genconj{(f + \indicator{\cV})}(\varphi)
		.
	\end{equation*}
\end{proposition}
\begin{proof}
	Due to \eqref{eq:FDconstant}, we may shift $\varphi$ by a suitable constant and thus assume \wolog $\genconj{(f + \indicator{\cV})}(\varphi) = 0$.
	By \eqref{eq:elementary-properties:3}, this means $\varphi \le f$ on~$\cV$.
	By definition of the Fenchel conjugate as a supremum, there exists $x_\varepsilon \in \cV$ such that $\varphi(x_\varepsilon) > f(x_\varepsilon) - \varepsilon/2$.

	The function $g \coloneqq \min \set{f, \varphi + \indicator{\overline \cV}}$ belongs to $\lscextF{\cM}$ as a minimum of lower semi-continuous functions.
	Moreover, we have $g = \varphi$ on $\cV$, so $g \in \lscextF{\cM} \cap C^k(\cV)$.

	By \cref{corollary:Ck-sandwich-lemma} we obtain $\psi \le g$ on~$\cM$ with $\psi = g = \varphi$ on $\cA$ and $\psi(x_\varepsilon) > g(x_\varepsilon)-\varepsilon/2 > f(x_\varepsilon) - \varepsilon$, hence $\genconj f(\psi) \ge -\varepsilon$.
	Since $\psi \le f$ on~$\cM$, we conclude $\genconj f(\psi) \le 0$.
\end{proof}

\subsection{Conjugates and Subdifferentials}%
\label{subsection:nonlinear-Fenchel-conjugates:manifolds:conjugates-and-subdifferentials}

In convex analysis on a locally convex topological space~$V$, the Fenchel conjugate is closely related to the convex subdifferential.
The latter can be defined, \eg, for $f \in \lscextF{V}$, as follows:
\begin{equation}
	\label{eq:convex-subdifferential}
	\partial f(x)
	\coloneqq
	\setDef[big]{x^* \in \dualSpace{V}}{f-x^* \text{ attains its (global) minimum at } x}
	\subseteq
	\dualSpace{V}
	.
\end{equation}
By convention, $\partial f(x) \coloneqq \emptyset$ when $f(x) = +\infty$.
It is a well-known fact of convex analysis that $x^* \in \partial f(x)$ holds if and only if the Fenchel-Young equality~\eqref{eq:FYE:1:1} holds for $(x,x^*,f)$, \ie, $\conjugate{f}(x^*) = x^*(x) - f(x)$.
In other words, the supremum in the definition of the classical Fenchel conjugate is attained at~$x$.

The convex subdifferential has been generalized in various ways for functions more general than convex.
One notion that is particularly suited for a generalization to manifolds, is that of the (viscosity) Fréchet subdifferential.
It is obtained from \eqref{eq:convex-subdifferential} by replacing the linear function~$x^*$ by a $C^1$-function~$\varphi$, and by replacing the global minimality by local minimality; see for instance \cite[Definition~3.1.2]{BorweinZhu:2005:1}:

\begin{definition}
	\label{definition:viscosity-Frechet-subdifferential}
	Suppose that $\cM$ is a $C^1$-Banach manifold.
	Moverover, suppose that $f \in \lscextF{\cM}$, $x \in \cM$ and $f(x) \neq +\infty$.
	The (viscosity) Fréchet subdifferential $\partial_F f(x)$ of~$f$ is defined as follows:
	\begin{equation*}
		\partial_F f(x)
		\coloneqq
		\setDef[big]{\varphi'(x)}{\varphi \in C^1(\cM), \; f-\varphi \text{ attains a local minimum at } x}
		\subseteq
		\cotangentSpace{x}[\cM]
		,
	\end{equation*}
	where $\cotangentSpace{x}[\cM] \coloneqq \dualSpace{(\tangentSpace{x}[\cM])}$ denotes the cotangent space at~$x$.
	In case $f(x) = +\infty$, we set $\partial_F f(x) \coloneqq \emptyset$.
\end{definition}

If $\cM$ is a linear space and $f \in \lscextF{\cM}$, then obviously $\partial f(x) \subseteq \partial_F f(x)$ holds.
Moreover, for $g \in C^1(\cM)$, we have
\begin{equation*}
	\partial_F(f + g)(x)
	 =
	\partial_F f(x) + \set{g'(x)}
	.
\end{equation*}
This is due to $f - \varphi = f + g - (\varphi + g)$ and $(\varphi + g)'(x) = \varphi'(x) + g'(x)$.
Furthermore, for $A \in C^1(\cM,\cN)$, $g \in \realF(\cN)$ and $x^* \in (\partial_F g)(A(x))$, we have $\adj{A'(x)} x^* \in \partial_F (g \circ A)(x)$, because $\varphi \circ A$ is a minorant of $g \circ A$ whenever $\varphi$ is a minorant of $g$.
That is,
\begin{equation*}
	\adj{A'(x)}\partial_F g(A(x))
	\subseteq
	\partial_F(g \circ A)(x)
	.
\end{equation*}

With \cref{definition:viscosity-Frechet-subdifferential} we can generalize the Fenchel-Young equality \eqref{eq:FYE:1:1}, which relates the classical Fenchel conjugate $\conjugate{f}$ and the subdifferential $\partial f$, to the nonlinear Fenchel conjugate $\genconj f$ and the viscosity Fréchet subdifferential $\partial_F f$ on manifolds:
\begin{proposition}
	\label{proposition:Fenchel-Young-equality-and-viscosity-Frechet-subdifferential}
	Suppose that $\cM$ is a $C^1$-Banach manifold.
	Moreover, consider $x \in \cM$, $f \in \lscextF{\cM}$ and $\varphi \in C^1(\cM)$.
	\begin{enumerate}
		\item \label[statement]{item:Fenchel-Young-equality-and-viscosity-Frechet-subdifferential:1}
			If $\genconj{f}(\varphi) = \varphi(x) - f(x)$, then $\varphi'(x) \in \partial_F f(x)$ and $\delta_x \in \partial (\restr{\genconj{f}}{C^1(\cM)})(\varphi)$.

		\item \label[statement]{item:Fenchel-Young-equality-and-viscosity-Frechet-subdifferential:2}
			Conversely, if $\delta_x \in \partial (\restr{\genconj{f}}{C^1(\cM)})(\varphi)$, then $\genconj{f}(\varphi) = \varphi(x) - f(x)$.
	\end{enumerate}
\end{proposition}
\begin{proof}
	\Cref{item:Fenchel-Young-equality-and-viscosity-Frechet-subdifferential:1} is a direct consequence of the definitions: \eqref{eq:FYE:1:1}, \ie, $\genconj{f}(\varphi) = \varphi(x) - f(x)$, means that $f - \varphi$ attains a (global) minimum at $x$, so $\varphi'(x)\in \partial_F f(x)$ by definition.
	On the other hand, $\genconj{f}(\varphi) = \varphi(x) - f(x) = \delta_x(\varphi) - \delta_x(f)$ also means that the restriction of $\genconj{f} - \delta_x$ to $C^1(\cM)$ attains its global minimum at~$\varphi$, so $\delta_x \in \partial (\restr{\genconj{f}}{C^1(\cM)})(\varphi)$ by definition.

	Notice that $\genconj{f}$ is a convex function on the linear space $C^1(\cM)$ by \cref{proposition:monotonicity-Lipschitz-continuity-convexity}~\ref{item:monotonicity-Lipschitz-continuity-convexity:3}, and thus the use of the convex subdifferential $\partial (\restr{\genconj{f}}{C^1(\cM)})$ is appropriate.

	For \cref{item:Fenchel-Young-equality-and-viscosity-Frechet-subdifferential:2}, suppose that $\delta_x \in \partial (\restr{\genconj{f}}{C^1(\cM)})(\varphi)$ holds.
	Then the restriction of $\genconj{f} - \delta_x$ to $C^1(\cM)$ attains its minimum at~$\varphi$, which implies $\genconj{f}(\varphi) = \varphi(x) - f(x)$.
\end{proof}

We point out that $\varphi'(x) \in \partial_F f(x)$ does \emph{not} in general imply the Fenchel-Young equality $\genconj{f}(\varphi) = \varphi(x) - f(x)$.
This is because $\varphi'(x) \in \partial_F f(x)$ only entails that $x$ is a \emph{local} minimizer of $f - \psi$, for some $\psi \in \cF$ with $\psi'(x) = \varphi'(x)$.
By contrast, the Fenchel-Young equality requires $x$ to be a \emph{global} minimizer of $f - \varphi$.

However, as shown in the following result, given $x^* \in \partial_F f(x)$, it is often possible to find $\varphi \in C^1(\cM)$ such that $x^* = \varphi'(x)$ holds and the Fenchel-Young equality is fulfilled:
\begin{proposition}
	\label{proposition:converse-of-Fenchel-Young-equality}
	Suppose that $\cM$ is a $C^1$-Banach manifold.
	\begin{enumerate}
		\item \label[statement]{item:converse-of-Fenchel-Young-equality:1}
			Let $x \in \cM$ and $f \in \lscextF{\cM}$.
			Suppose that $x^*\in \partial_F f(x)$.
			Then there exist $\varphi \in C^1(\cM)$ and an open neighborhood $\cV$ of $x$ such that $\varphi'(x) = x^*$ and the Fenchel-Young equality~\eqref{eq:FYE:1:1} holds for $(x,\restr{\varphi}{\cV}, \restr{f}{\cV})$, \ie,
			\begin{equation*}
				\genconj{(f + \indicator{\cV})}(\varphi)
				=
				\varphi(x)
				-
				f(x)
				.
			\end{equation*}

		\item \label[statement]{item:converse-of-Fenchel-Young-equality:2}
			Assume, in addition, that $\cM$ admits $C^1$-partitions of unity.
			Then there exists $\psi \in C^1(\cM)$ such that $\psi'(x) = x^*$ and the Fenchel-Young equality~\eqref{eq:FYE:1:1} holds for $(x,\psi,f)$, \ie,
			\begin{equation*}
				\genconj{f}(\psi)
				=
				\psi(x)
				-
				f(x)
				.
			\end{equation*}
	\end{enumerate}
\end{proposition}
\begin{proof}
	\Cref{item:converse-of-Fenchel-Young-equality:1}:
	By definition of $\partial_F$ there exists $\varphi \in C^1(\cM)$ with $x^* = \varphi'(x)$ such that $f - \varphi$ attains a local minimum at~$x$.
	Let $\cV \subseteq \cM$ be a neighborhood such that this minimum is global on~$\cV$.
	Then clearly~\eqref{eq:FYE:1:1} holds, restricted to~$\cV$.

	\Cref{item:converse-of-Fenchel-Young-equality:2}:
	We may apply \cref{proposition:extension-of-local-Fenchel-conjugates} to $\varphi$ with some closed neighborhood $\cA \subseteq \cV$ of~$x$ to obtain a function~$\psi$ with the property $\psi = \varphi$ on~$\cA$, implying $\psi'(x) = \varphi'(x)$.
	Furthermore,
	\begin{equation*}
		\genconj{f}(\psi)
		\le
		\genconj{(f + \indicator{\cV})}(\varphi)
		=
		\varphi(x) - f(x)
		=
		\psi(x) - f(x)
		\le
		\genconj{f}(\psi)
		.
	\end{equation*}
	The first inequality is due to \cref{proposition:extension-of-local-Fenchel-conjugates} and the second inequality follows from the definition of the Fenchel conjugate.
\end{proof}

\subsection{Restriction to Test Functions Generated by the Exponential Map}%
\label{subsection:nonlinear-Fenchel-conjugates:manifolds:comparison}

In this subsection, we further restrict the choice of (smooth) test functions compared to \cref{subsection:nonlinear-Fenchel-conjugates:manifolds:smooth-functions}.
In contrast to the infinite-dimensional linear space $\cF = C^k(\cM)$ of test functions, we now consider a linear space~$\cF$ of the same dimension as~$\cM$.

To this end, suppose that $\cM$ is a finite-dimensional \emph{Riemannian} $C^\infty$-manifold.
Consider a point $x \in \cM$ and a neighborhood~$V$ of the origin in the tangent space $\tangentSpace{x}$ such that the exponential map~$\exponential{x}$ is a diffeomorphism $V \to \cV \coloneqq \exponential{x}(V) \subseteq \cM$.

We define the restricted function space
\begin{equation*}
	\cF_x
	\coloneqq
	\setDef[big]{x^*\circ \exponential{x}^{-1} \in C^\infty(\cV,\R)}{x^* \in \cotangentSpace{x}}
\end{equation*}
obtained by composition of linear functions on $\tangentSpace{x}$ with the inverse exponential map.
$\cF_x$ is a linear space of dimension~$\dim(\cM)$.
Since elements of $\cF_x$ are only defined locally, we proceed as in \eqref{eq:local-Fenchel-conjugate:1} and consider, for $f \in \ExtF{\cM}$,
\begin{equation}
	\label{eq:local-Fenchel-conjugate:2}
	\genconj{(f + \indicator{\cV})}(\varphi)
	=
	\sup_{y \in \cV} \paren[big]\{\}{\varphi(y) - f(y)}
	\quad
	\text{for }
	\varphi
	\in
	\cF_x
	.
\end{equation}

We now show that \eqref{eq:local-Fenchel-conjugate:2} agrees with the classical Fenchel conjugate of $f \circ \exponential{x} + \indicator{V}$, evaluated at~$x^*$.
To this end, consider $\varphi \in \cF_x$ such that $\varphi = x^*\circ \exponential{x}^{-1}$.
We evaluate
\begin{subequations}
	\begin{align}
		\conjugate{(f \circ \exponential{x} + \indicator{V})}(x^*)
		&
		=
		\sup \setDef[big]{x^*(\xi) - f(\exponential{x}(\xi)) - \indicator{V}(\xi)}{\xi \in \tangentSpace{x}}
		\notag
		\\
		&
		=
		\sup \setDef[big]{x^*(\xi) - f(\exponential{x}(\xi))}{\xi \in V}
		\label{eq:local-Fenchel-conjugate:3}
		\\
		&
		=
		\sup \setDef[big]{x^*(\exponential{x}^{-1}(y)) - f(y)}{y \in \cV}
		\label{eq:local-Fenchel-conjugate:4}
		\\
		&
		=
		\sup \setDef[big]{\varphi(y) - f(y)}{y \in \cV}
		\notag
		\\
		&
		=
		\genconj{(f + \indicator{\cV})}(\varphi)
		.
		\notag
	\end{align}
\end{subequations}
We are now in the position to compare $\genconj{(f + \indicator{\cV})}(\varphi)$ with the concept of $x$-Fenchel conjugate from \cite[Definition~3.1]{BergmannHerzogSilvaLouzeiroTenbrinckVidalNunez:2021:1}.
As can be seen by comparison with \eqref{eq:local-Fenchel-conjugate:3}, that definition coincides with $\conjugate{(f \circ \exponential{x} + \indicator{V})}$ but was only defined for strongly geodesically convex sets~$\cV$.
In case $\cM$ is a Hadamard manifold, one can choose $\cV = \cM$ and we obtain $V = \tangentSpace{x}$.
Then $\genconj{(f + \indicator{\cV})}(\varphi)$ also agrees with the $x$-dual function from \cite[Definition~3.1]{AhmadiKakavandiAmini:2010:1}.
As a variant, \cite[Definition~3.1]{SilvaLouzeiroBergmannHerzog:2022:1} use \eqref{eq:local-Fenchel-conjugate:4} on Hadamard manifolds with $\cV = \cM$ and $V = \tangentSpace{x}$, amounting to a definition of a Fenchel conjugate on the entire cotangent bundle.

\section{Nonlinear Fenchel Conjugates on Groups}%
\label{section:nonlinear-Fenchel-conjugates:groups}

This section is concerned with the nonlinear Fenchel conjugate on groups.
This is a natural setting to consider convolutions of extended real-valued functions.
Restricting the test functions to group homomorphisms into the additive reals, we show that the infimal convolution formula holds, \ie, the conjugate of an infimal convolution is the sum of the conjugates (\cref{subsection:nonlinear-Fenchel-conjugates:nonlinear-infimal-convolution}).
Moreover, replacing group homomorphisms by continuous affine functions, we obtain a notion of convex functions on groups through their $\sup$-closure (\cref{subsection:nonlinear-Fenchel-conjugates:groups:convexity}).

\subsection{Nonlinear Infimal Convolution}%
\label{subsection:nonlinear-Fenchel-conjugates:nonlinear-infimal-convolution}

The infimal convolution $f \infconvolution g$ for functions $f \in \extF{V}$ on a linear space~$V$ is defined as
\begin{equation}
	\label{eq:infConvolution}
	(f \infconvolution g)(x)
	\coloneqq
	\inf_{y \in V} f(x-y) + g(y)
	=
	\inf_{z \in V} f(z) + g(x-z)
	.
\end{equation}
The infimal convolution formula (see, \eg, \cite[Proposition~13.21]{BauschkeCombettes:2011:1}) shows that
\begin{equation}
	\label{eq:FenchelInfPlus}
	\conjugate{(f \infconvolution g)}
	=
	\conjugate{f} + \conjugate{g}
\end{equation}
holds for $f, g \in \extF{V}$.

We now generalize the infimal convolution \eqref{eq:infConvolution} to functions $f, g \in \extF{\cM}$ on a group $(\cM,\cdot)$, compare \cite{Bachir:2015:1}, by setting
\begin{equation}
	\label{eq:infConvolution-group}
	(f \infconvolution g)(x)
	\coloneqq
	\inf_{y \in \cM} f(x \cdot y^{-1}) + g(y)
	=
	\inf_{z \in \cM} f(z) + g(z^{-1} \cdot x)
	.
\end{equation}
We denote by
\begin{equation*}
	\cH
	\coloneqq
	\Homo((\cM,\cdot),(\R,+))
\end{equation*}
the linear space of group homomorphisms $(\cM,\cdot) \to (\R,+)$.
We are now in the position to generalize the infimal convolution formula \eqref{eq:FenchelInfPlus} to the following result.
\begin{proposition}
	\label{proposition:FenchelInfPlus}
	Suppose that $(\cM,\cdot)$ is a group and $f, g \in \extF{\cM}$.
	Then the conjugate of the infimal convolution $f \infconvolution g$ equals the sum of the conjugates when restricted to~$\cH = \Homo((\cM,\cdot),(\R,+))$, \ie,
	\begin{equation}
		\label{eq:FenchelInfPlus-groups}
		\genconj{(f \infconvolution g)}(\varphi)
		=
		\genconj{f}(\varphi)
		+
		\genconj{g}(\varphi)
		\quad
		\text{for all }
		\varphi
		\in
		\cH
		.
	\end{equation}
\end{proposition}
\begin{proof}
	For any $\varphi \in \cH$, we find
	\begin{align*}
		\genconj{f}(\varphi) + \genconj{g}(\varphi)
		&
		=
		\sup_{z \in \cM} \paren[big]\{\}{\varphi(z) - f(z)}
		+
		\sup_{y \in \cM} \paren[big]\{\}{\varphi(y) - g(y)}
		\\
		&
		=
		\sup \setDef[big]{\varphi(z) + \varphi(y) - [f(z) + g(y)]}{y, z \in \cM}
		\\
		&
		=
		\sup \setDef[big]{\varphi(z \cdot y) - [f(z) + g(y)]}{y, z \in \cM}
		\quad
		\text{due to $\varphi \in \cH$}
		\\
		&
		=
		\sup \setDef[big]{\varphi(x) - [f(x \cdot y^{-1}) + g(y)]}{x, y \in \cM}
		\\
		&
		=
		\sup \setDef[big]{\varphi(x) - (f \infconvolution g)(x)}{x \in \cM}
		\\
		&
		=
		\genconj{(f \infconvolution g)}(\varphi)
		.
		\qedhere
	\end{align*}
\end{proof}

\subsection{A Notion of Convexity on Groups}%
\label{subsection:nonlinear-Fenchel-conjugates:groups:convexity}

In this subsection we introduce a notion of (mid-point) convexity for extended real-valued functions on the group~$(\cM,\cdot)$.
For $\cH = \Homo((\cM,\cdot),(\R,+))$, we begin by studying $\supcl{\cH}$, the set of pointwise suprema of subsets of~$\affine{\cH}$.
We will see that this set plays the role of the class of \lsc convex functions in our nonlinear setting.

To this end, consider the embedding \eqref{eq:canemb}
\begin{equation*}
	\bidualembedding{\cM}{\algebraicdualSpace{\cH}}
	\colon
	\cM
	\ni
	x
	\mapsto
	\delta_x
	\in
	\algebraicdualSpace{\cH}
\end{equation*}
that maps $x \in \cM$ to the evaluation function $\delta_x \in \algebraicdualSpace{\cH}$ so that we have $\bidualembedding{\cM}{\algebraicdualSpace{\cH}}(x)(\varphi) = \delta_x(\varphi) = \varphi(x)$ for $\varphi \in \cH$.

\begin{lemma}
	\label{lem:JHhom}
	Suppose that $(\cM,\cdot)$ is a group and $\cH = \Homo((\cM,\cdot),(\R,+))$.
	Then the embedding $\bidualembedding{\cM}{\algebraicdualSpace{\cH}} \colon (\cM,\cdot) \to (\algebraicdualSpace{\cH},+)$ is a homomorphism.
\end{lemma}
\begin{proof}
	Suppose that $x, y \in \cM$ and $\varphi \in \cH$.
	Then we have
	\begin{equation*}
		\delta_x(\varphi) + \delta_y(\varphi)
		=
		\varphi(x) + \varphi(y)
		=
		\varphi(x \cdot y)
		=
		\delta_{x \cdot y}(\varphi)
	\end{equation*}
	and thus $\bidualembedding{\cM}{\algebraicdualSpace{\cH}}(x) + \bidualembedding{\cM}{\algebraicdualSpace{\cH}}(y) = \delta_x + \delta_y = \delta_{x \cdot y} = \bidualembedding{\cM}{\algebraicdualSpace{\cH}}(x \cdot y)$.
\end{proof}

We work under the assumption that the linear space $\cH = \Homo((\cM,\cdot),(\R,+))$ is \emph{finite-dimensional}.
\begin{remark}
	For comparison, suppose that $(\cM,\cdot)$ is the additive Abelian group $(V,+)$ of a finite-dimensional real linear space~$V$ endowed with the Euclidean topology.
	We modify $\cH$ to denote the linear space of all \emph{continuous} additive homomorphisms from $(\cM,\cdot)$ into $(\R,+)$.
	Cauchy proved that continuous additive functions are linear; see, for instance, \cite{Hamel:1905:1}.
	We thus obtain $\cH = \dualSpace{V}$.
	In this setting, the pointwise suprema of subsets of~$\affine{\cH}$ are precisely the convex and \lsc functions on $\dualSpace{\cH} = \bidualSpace{V}$, which is linearly and topologically isomorphic to~$V$; see \cref{theorem:characterization-of-sup-closure}.
\end{remark}

The following two lemmas provide sufficient conditions for $\cH$ to be finite-dimen\-sio\-nal.
\begin{lemma}
	\label{lemma:finitely-generated-group}
	Consider a group $(\cM,\cdot)$, its commutator $[\cM,\cM]$ and the linear space $\cH = \Homo((\cM,\cdot),(\R,+))$.
	Suppose that the $\cM$ or its Abelianization $\abelianization{\cM} \coloneqq \quotient{\cM}{[\cM,\cM]}$ are finitely generated.
	Then $\cH$ is finite-dimensional.
\end{lemma}
\begin{proof}
	When $\cM$ is finitely generated, then so is $\abelianization{\cM}$, and thus we only consider the latter case.

	Since the codomain $(\R,+)$ is Abelian, the linear spaces $\cH = \Homo((\cM,\cdot),(\R,+))$ and $\widehat \cH \coloneqq \Homo((\abelianization{\cM},\cdot),(\R,+))$ are isomorphic as groups; see for instance \cite[Problem~167, p.60]{Rose:1994:1}.
	It is easy to verify that they are also isomorphic as linear spaces.

	Suppose that $(\abelianization{\cM},\cdot)$ is finitely generated and $F \coloneqq \family[big]{\equivalenceclass{x_1}, \ldots, \equivalenceclass{x_n}}$ is a generating family with $x_i \in \cM$ and cosets $\equivalenceclass{x_i}$ \wrt the commutator.
	Then any $\varphi \in \widehat \cH$ is uniquely determined by its values on~$F$.
	Consequently, there is a monomorphism $I \colon \widehat \cH \ni \varphi \mapsto \varphi(F) \in \R^n$.
	Due to potential relations present among the generators, the image of~$I$ is a subspace $U \subseteq \R^n$.
	In any case, $\cH$ and $\widehat \cH$ are both isomorphic to~$U$ and thus of finite dimension.
\end{proof}

\begin{lemma}
	\label{lemma:Lie-group}
	Consider a Lie group $(\cM,\cdot)$ with Lie algebra $(\fm,[\cdot,\cdot])$, and commutator
	\begin{equation*}
		[\fm,\fm]
		\coloneqq
		\setDef{[v,w]}{ v, w \in \fm}
		.
	\end{equation*}
	Suppose that $\quotient{\fm}{[\fm,\fm]}$ is finite dimensional and $\cM$ is connected.
	Then $\cH$ is finite-dimensional and $\bidualembedding{\cM}{\algebraicdualSpace{\cH}}$ is surjective.
\end{lemma}
\begin{proof}
	Consider the Lie group $(\cM,\cdot)$ with neutral element~$e$ and its exponential map $\exponential{e} \colon \fm \to \cM$, as well as $(\R,+)$ with neutral element~$0$ and $\exponential{0} = \id_{\R}$.
	Let $\varphi \in \cH$.
	By the commutativity of homomorphisms and exponential maps, we obtain
	\begin{equation*}
		\varphi'(e) \, v
		=
		\exponential{0}(\varphi'(e) \, v)
		=
		\varphi(\exponential{e}(v))
		\quad
		\text{for all }
		v \in \fm
		.
	\end{equation*}
	Thus, for any $v, w \in \fm$:
	\begin{equation}
		\label{eq:exphom}
		\varphi(\exponential{e}(v)\exponential{e}(w))
		=
		\varphi(\exponential{e}(v)) + \varphi(\exponential{e}(w))
		=
		\varphi'(e)(v+w)
		.
	\end{equation}
	Since for $\varphi \in \cH$, its derivative $\varphi'(e)$ is a Lie algebra homomorphism (\cite[Theorem~8.44]{Lee:2012:1}), we obtain in particular that
	\begin{equation*}
		\varphi'(e) \lie{v}{w}
		=
		\lie{\varphi'(e) \, v}{\varphi'(e) \, w}
		=
		0
	\end{equation*}
	holds, since $(\R,+)$ is Abelian and thus its Lie bracket $\lie{\cdot}{\cdot}$ vanishes.
	Denoting the linear space of these Lie algebra homomorphisms by
	\begin{equation*}
		\fh
		\coloneqq
		\setDef[big]{\ell \in \algebraicdualSpace{\fm}}{\ell([v,w]) = 0 \text{ for all } v, w \in \fm}
		\cong
		\algebraicdualSpace{(\fm/[\fm,\fm])}
	\end{equation*}
	we thus obtain a linear mapping
	\begin{equation}
		\label{eq:linmapH}
			H
			\colon
			\cH
			\ni
			\varphi
			\mapsto
			\varphi'(e)
			\in
			\fh.
	\end{equation}
	It follows from Lie theory that $H$ is injective when $\cM$ is connected, because then every $x \in \cM$ can be written as a product of factors $\exponential{e}(v_i)$ with $v_i \in \fm$.
	Due to \cref{eq:exphom} this implies that $\varphi$ is given uniquely by its values of $\varphi'$ on $\fm$.
	Thus, by the injectivity of~$H$, $\dim \cH \le \dim \fh = \dim \quotient{\fm}{[\fm,\fm]}$.

	Concerning the surjectivity of $\bidualembedding{\cM}{\dualSpace{\cH}}$, we compute
	\begin{equation*}
		\begin{aligned}
			\bidualembedding{\cM}{\dualSpace{\cH}}(\exponential{e}(v))(\varphi)
			&
			=
			\varphi(\exponential{e}(v))
			=
			\varphi'(e)v
			\\
			&
			=
			(H \varphi)(v)
			=
			\adj{H} v(\varphi)
			\quad
			\text{for all }
			v \in \fm
			,
			\varphi \in \cH
			.
		\end{aligned}
	\end{equation*}
	Since $H$ is injective, $\adj{H}$ is surjective, and thus $\bidualembedding{\cM}{\dualSpace{\cH}} \circ \exponential{e}$ is surjective as well.
	This in turn implies the surjectivity of $\bidualembedding{\cM}{\dualSpace{\cH}}$.
\end{proof}

Let us now consider the dual map
\begin{align*}
	\genadj{\bidualembedding{\cM}{\algebraicdualSpace{\cH}}}
	\colon
	\ExtF{\algebraicdualSpace{\cH}}
	&
	\to
	\ExtF{\cM}
	\\
	f
	&
	\mapsto
	\genadj{\bidualembedding{\cM}{\algebraicdualSpace{\cH}}}(f)
	\coloneqq
	f \circ \bidualembedding{\cM}{\algebraicdualSpace{\cH}}
	.
\end{align*}
That is, for $f \in \ExtF{\algebraicdualSpace{\cH}}$, we have
\begin{equation*}
	\genadj{\bidualembedding{\cM}{\algebraicdualSpace{\cH}}}(f)(x)
	=
	f(\bidualembedding{\cM}{\algebraicdualSpace{\cH}}(x))
	=
	f(\delta_x)
	\quad
	\text{for }
	x \in \cM
	.
\end{equation*}
We now restrict $\genadj{\bidualembedding{\cM}{\algebraicdualSpace{\cH}}}$ from $\ExtF{\algebraicdualSpace{\cH}}$ to $\algebraicbidualSpace{\cH}$, \ie, to \emph{linear real-valued} functions on~$\algebraicdualSpace{\cH}$.

\begin{lemma}
	\label{lem:emdAdj}
	Suppose that $(\cM,\cdot)$ is a group and $\cH = \Homo((\cM,\cdot),(\R,+))$ is finite-dim\-en\-sio\-nal.
	Then $\restr{\genadj{\bidualembedding{\cM}{\algebraicdualSpace{\cH}}}}{\algebraicbidualSpace{\cH}} \colon \algebraicbidualSpace{\cH} \to \cH$ is an isomorphism of linear spaces.
	Therefore,
	\begin{equation*}
		\cH
		=
		\genadj{\bidualembedding{\cM}{\algebraicdualSpace{\cH}}}(\algebraicbidualSpace{\cH})
		=
		\setDef{\ell \circ \bidualembedding{\cM}{\algebraicdualSpace{\cH}}}{\ell \in \algebraicbidualSpace{\cH}}
		.
	\end{equation*}
	The inverse of $\restr{\genadj{\bidualembedding{\cM}{\algebraicdualSpace{\cH}}}}{\algebraicbidualSpace{\cH}}$ is the canonical embedding $\embedding{\cH}{\algebraicbidualSpace{\cH}} \colon \cH \to \algebraicbidualSpace{\cH}$.
\end{lemma}
\begin{proof}
	To show $\genadj{\bidualembedding{\cM}{\algebraicdualSpace{\cH}}}(\algebraicbidualSpace{\cH})\subseteq \cH$, we observe that both $\bidualembedding{\cM}{\algebraicdualSpace{\cH}} \colon (\cM,\cdot) \to (\algebraicdualSpace{\cH},+)$ and $\ell \colon (\algebraicdualSpace{\cH},+) \to (\R,+)$ are homomorphisms and so is their composition $\genadj{\bidualembedding{\cM}{\algebraicdualSpace{\cH}}} \circ \ell = \ell \circ \bidualembedding{\cM}{\algebraicdualSpace{\cH}} \in \Homo((\cM,\cdot),(\R,+)) = \cH$.
	Thus $\restr{\genadj{\bidualembedding{\cM}{\algebraicdualSpace{\cH}}}}{\algebraicbidualSpace{\cH}} \colon \algebraicbidualSpace{\cH}\to \cH$ is well-defined.

	For the proof of bijectivity of $\restr{\genadj{\bidualembedding{\cM}{\algebraicdualSpace{\cH}}}}{\algebraicbidualSpace{\cH}}$, consider the canonical embedding $\embedding{\cH}{\algebraicbidualSpace{\cH}} \colon \cH \to \algebraicbidualSpace{\cH}$ of~$\cH$ into its bidual space.
	Both $\restr{\genadj{\bidualembedding{\cM}{\algebraicdualSpace{\cH}}}}{\algebraicbidualSpace{\cH}}$ and $\embedding{\cH}{\algebraicbidualSpace{\cH}}$ are linear maps between linear spaces $\cH$ and $\algebraicbidualSpace{\cH}$ of finite and equal dimension.
	It is therefore enough to verify that $\genadj{\bidualembedding{\cM}{\algebraicdualSpace{\cH}}} \circ \embedding{\cH}{\algebraicbidualSpace{\cH}} = \id_\cH$
	holds.
	Indeed, for $\varphi \in \cH$ and $x \in \cM$ we evaluate
	\begin{align*}
		\paren[big][]{(\genadj{\bidualembedding{\cM}{\algebraicdualSpace{\cH}}} \circ \embedding{\cH}{\algebraicbidualSpace{\cH}})(\varphi)}(x)
		&
		=
		\paren[big][]{(\embedding{\cH}{\algebraicbidualSpace{\cH}}(\varphi)) \circ \bidualembedding{\cM}{\algebraicdualSpace{\cH}}}(x)
		\\
		&
		=
		\paren[big][]{\embedding{\cH}{\algebraicbidualSpace{\cH}}(\varphi)}(\delta_x)
		=
		\delta_x(\varphi)
		=
		\varphi(x)
		.
		\qedhere
	\end{align*}
\end{proof}
It is clear that when $\cH$ is endowed with the standard Euclidean topology, then the algebraic and topological dual spaces agree, as do the respective bidual spaces, \ie, we have $\algebraicdualSpace{\cH} = \dualSpace{\cH}$ and $\algebraicbidualSpace{\cH} = \bidualSpace{\cH}$.

We proceed to characterize the $\sup$-closure of $\cH$.
\begin{theorem}
	\label{theorem:characterization-of-sup-closure}
	Suppose that $(\cM,\cdot)$ is a group and $\cH = \Homo((\cM,\cdot),(\R,+))$ is finite-dim\-en\-sio\-nal and endowed with the standard Euclidean topology.
	The $\sup$-closure of $\cH \subseteq \ExtF{\cM}$ is related to the set of convex \lsc functions on $\dualSpace{\cH}$ in the following way:
	\begin{equation*}
		\begin{aligned}
			\supcl{\cH}
			&
			=
			\genadj{\bidualembedding{\cM}{\dualSpace{\cH}}}(\supcl{\bidualSpace{\cH}})
			\\
			&
			=
			\setDef[big]{f \circ \bidualembedding{\cM}{\dualSpace{\cH}}}{f \colon \dualSpace{\cH} \to \extR \text{ is convex and \lsc}} \cup \set{-\infty}
			.
		\end{aligned}
	\end{equation*}
\end{theorem}
\begin{proof}
	From \cref{lem:emdAdj} and \cref{proposition:pullback} and the fact that $\algebraicdualSpace{\cH} = \dualSpace{\cH}$ and $\algebraicbidualSpace{\cH} = \bidualSpace{\cH}$ hold, we conclude
	\begin{equation*}
		\supcl{\cH}
		=
		\supcl{\genadj{\bidualembedding{\cM}{\dualSpace{\cH}}}(\bidualSpace{\cH})}
		=
		\genadj{\bidualembedding{\cM}{\dualSpace{\cH}}}(\supcl{\bidualSpace{\cH}})
		.
	\end{equation*}
	By \eqref{eq:sup-closure:2}, $\supcl{\bidualSpace{\cH}}$ is the set of functions which can be written as pointwise suprema over collections of affine functions on $\dualSpace{\cH}$.
	That is, $\supcl{\bidualSpace{\cH}}$ is the set of \lsc convex functions on $\dualSpace{\cH}$; see \cref{example:sup-closures}.
\end{proof}

A subset $\cK$ of a linear space is called mid-point convex if, for any $x, y \in \cK$, also $\frac{1}{2}(x+y) \in \cK$.
A function $f \in \ExtF{\cK}$ is called mid-point convex if its epigraph is mid-point convex.
\begin{proposition}
	\label{proposition:convexity}
	Suppose that $(\cM,\cdot)$ is a group and $\cH = \Homo((\cM,\cdot),(\R,+))$ is finite-dimensional.
	Consider functions $f \colon \algebraicdualSpace{\cH} \to \extR$ and $\psi = f \circ \bidualembedding{\cM}{\algebraicdualSpace{\cH}}$.
	\begin{enumerate}
		\item \label[statement]{item:convexity:1}
			Suppose that $f \colon \algebraicdualSpace{\cH} \to \extR$ is mid-point convex.
			Then
			\begin{equation}
				\label{eq:groupConvex}
				\psi(x)
				\le
				\frac{1}{2} \paren[big](){\psi(y \cdot x) + \psi(y^{-1} \cdot x)}
				\quad
				\text{for all }
				x, y \in \cM
				.
			\end{equation}

		\item \label[statement]{item:convexity:2}
			Conversely, if $\cK \coloneqq \range \bidualembedding{\cM}{\algebraicdualSpace{\cH}}$ is mid-point convex and $\psi$ satisfies~\eqref{eq:groupConvex}, then $\restr{f}{\cK}$ is mid-point convex.
	\end{enumerate}
\end{proposition}
\begin{proof}
	Since $\bidualembedding{\cM}{\algebraicdualSpace{\cH}}$ is a homomorphism by \cref{lem:JHhom}, inequality \eqref{eq:groupConvex} is equivalent to
	\begin{equation*}
		f(\delta_x)
		\le
		\frac{1}{2} \paren[big](){f(\delta_x + \delta_y) + f(\delta_x - \delta_y)}
		\quad
		\text{for all }
		x, y \in \cM
		.
	\end{equation*}
	This is in turn equivalent to
	\begin{equation*}
		f(\ell_1)
		\le
		\frac{1}{2} \paren[big](){f(\ell_1 + \ell_2) + f(\ell_1 - \ell_2)}
		\quad
		\text{for all }
		\ell_1, \ell_2 \in \cK
		.
	\end{equation*}
	Inserting into this inquality the functionals $\tilde \ell_1 \coloneqq \ell_1 + \ell_2$ and $\tilde \ell_2 \coloneqq \ell_1 - \ell_2$ so that $\ell_1 = \frac{1}{2}(\tilde \ell_1+\tilde \ell_2)$, we conclude that~\eqref{eq:groupConvex} holds if $f$ is mid-point convex (and vice-versa on $\cK$, if $\cK$ is mid-point convex):
	\begin{align*}
		f \paren[big](){ \frac{1}{2}(\tilde \ell_1 + \tilde \ell_2)}
		&
		\le
		\frac{1}{2} \paren[big](){f(\tilde \ell_1) + f(\tilde \ell_2)}
		\quad
		\text{for all }
		\tilde \ell_1, \tilde \ell_2 \in \algebraicdualSpace{\cH}
		.
		\qedhere
	\end{align*}
\end{proof}

We point out that \eqref{eq:groupConvex} can be seen as a definition for a function $\cM \to \extR$ to be mid-point convex on the group $(\cM,\cdot)$.
This notion agrees with the one introduced in \cite{JarczykLaczkovich:2009:1}.

\section{Conclusion}
\label{section:conclusion}

The classical Fenchel conjugate~$\conjugate{f}$ is defined for extended real-valued functions $f \in \ExtF{V}$ on a linear space~$V$, and it acts on \emph{linear} functions on~$V$.

In this paper we generalize the classical Fenchel conjugate to nonlinear spaces $\cM$, where linear functions do not exist.
This new concept of a nonlinear Fenchel conjugate~$\genconj{f}$ applies to extended real-valued functions $f \in \ExtF{\cM}$ on any set~$\cM$.
Several classical key results --- including the Fenchel-Young inequality and calculus rules --- carry over to the nonlinear Fenchel conjugate.
For the classical Fenchel conjugate, the concept of $\Gamma$-regularization is obtained by approximating a function by the supremum of its affine minorants.
For the nonlinear Fenchel conjugate, we generalise this to the so-called $\cF$-regularization using minorants more general than just affine functions.
In this setting, the biconjugation theorem carries over to the nonlinear Fenchel conjugate.

By considering differentiable manifolds in place of general sets~$\cM$, we are able to relate the nonlinear Fenchel conjugate with the viscosity Fréchet subdifferential.
In addition, we recover previous attempts to generalize Fenchel conjugates to manifolds.
Finally, when $\cM$ is a group, we found a generalization of the classical infimal convolution formula.

An interesting aspect for future work is how this general framework of nonlinear Fenchel conjugates can be used to derive optimization algorithms based on operator splittings similar to primal-dual hybrid gradient algorithms, see for instance~\cite{EsserZhangChan:2010:1}, or the Chambolle-Pock algorithm~\cite{ChambollePock:2011:1}.

\appendix
\section{Equivalence of \eqref{eq:general-nonlinear-Fenchel-conjugate}--\eqref{eq:general-nonlinear-Fenchel-conjugate:alternative:3}}
\label{section:equivalence-of-alternative-formulations-of-Fenchel-conjugate}

For convenience, we recall the equations under consideration:
\begin{align*}
	&
	\genconj{f}(\varphi)
	=
	\sup \setDef{\varphi(x) - f(x)}{x \in \domain{\varphi - f}}
	\tag{\ref{eq:general-nonlinear-Fenchel-conjugate}}
	\\
	&
	\begin{aligned}
		\genconj{f}(\varphi)
		&
		=
		\sup \setDef{\varphi(x) - f(x)}{x \in \rdomain{\varphi - f}}
		\\
		&
		=
		\sup \setDef{\varphi(x) - f(x)}{x \in \cM, \; \varphi(x) \neq -\infty \text{ and } f(x) \neq +\infty}
	\end{aligned}
	\tag{\ref{eq:general-nonlinear-Fenchel-conjugate:alternative:1}}
	\\
	&
	\genconj{f}(\varphi)
	=
	\inf \setDef{c \in \R}{\varphi(x) \le f(x)+c \text{ for all } x \in \cM}
	\tag{\ref{eq:general-nonlinear-Fenchel-conjugate:alternative:2}}
	\\
	&
	\genconj{f}(\varphi)
	=
	\sup \setDef{\varphi(x) - c}{(x,c) \in \epi f}
	\tag{\ref{eq:general-nonlinear-Fenchel-conjugate:alternative:3}}
\end{align*}
for $f, \varphi \in \ExtF{\cM}$.

\begin{proof}[Proof of equivalence]
	The conditions in \eqref{eq:general-nonlinear-Fenchel-conjugate:alternative:1} exclude points outside of $\domain{\varphi - f}$ and, in addition, points where $\varphi(x) - f(x) = -\infty$.
	This shows the equivalence between \eqref{eq:general-nonlinear-Fenchel-conjugate} and \eqref{eq:general-nonlinear-Fenchel-conjugate:alternative:1}.
	To prove the equivalence of \eqref{eq:general-nonlinear-Fenchel-conjugate:alternative:1} and \eqref{eq:general-nonlinear-Fenchel-conjugate:alternative:2}, we notice that we can exclude from \eqref{eq:general-nonlinear-Fenchel-conjugate:alternative:2} all points that impose no restrictions on the admissible values for~$c$.
	In other words, \eqref{eq:general-nonlinear-Fenchel-conjugate:alternative:2} is equal to $\inf \setDef{c \in \R}{\varphi(x) \le f(x) + c \text{ for all } x \in \rdomain{\varphi - f}}$.
	The equivalence between \eqref{eq:general-nonlinear-Fenchel-conjugate:alternative:1} and \eqref{eq:general-nonlinear-Fenchel-conjugate:alternative:2} is now straightforward.
	Finally we show the equivalence between \eqref{eq:general-nonlinear-Fenchel-conjugate} and \eqref{eq:general-nonlinear-Fenchel-conjugate:alternative:3}.
	If $f(x) = +\infty$, then $(x,c) \not \in \epi f$ for any $c \in \R$.
	If $\varphi(x) = -\infty$, then $\varphi(x)-c = -\infty$ for all $c \in \R$.
	Thus we can exclude all these $x$ from the supremum in~\eqref{eq:general-nonlinear-Fenchel-conjugate:alternative:3}, and we may write the right-hand side of \eqref{eq:general-nonlinear-Fenchel-conjugate:alternative:3} as
	\begin{equation*}
		\sup \setDef{\varphi(x) - c}{(x,c) \in \epi f}
		=
		\sup \setDef{\varphi(x) - c}{c \ge f(x) \text{ and } x \in \domain{\varphi - f}}
		.
	\end{equation*}
	For $x \in \domain{\varphi - f}$ we have that $\varphi(x) - f(x) = \sup \setDef{\varphi(x) - c}{c \ge f(x)}$ is well-defined, showing the desired equality.
\end{proof}

\printbibliography

\end{document}